\documentclass[11pt,leqno]{article}
\usepackage[doc]{optional}
\usepackage{enumitem}
\usepackage{multirow}

\usepackage{color}
\usepackage{float}
\usepackage{soul}
\usepackage{graphicx}
\definecolor{labelkey}{rgb}{0,0.08,0.45}
\definecolor{refkey}{rgb}{0,0.6,0.0}

\definecolor{Brown}{rgb}{0.45,0.0,0.05}
\definecolor{lime}{rgb}{0.00,0.8,0.0}
\definecolor{lblue}{rgb}{0.5,0.5,0.99}
\definecolor{lblue}{rgb}{0.8,0.85,1.00}
\definecolor{anotherblue}{rgb}{.8, .8,1}
\definecolor{violet}{rgb}{0.9,0.6,0.9}
\definecolor{greenyellow}{rgb}{0.53,0.99,0.18}

\definecolor{Lyellow}{rgb}{0.87,0.87,0.87}
\definecolor{Lgray}{rgb}{0.92,0.92,0.92}
\definecolor{Mgray}{rgb}{0.5,0.5,0.5}
\definecolor{Gold}{rgb}{0.99,0.84,0.0}

\usepackage{mathpazo}
\usepackage{amsmath}
\usepackage{amssymb}
\usepackage{theorem}
\usepackage{fancyhdr}
\usepackage{empheq}

\usepackage[margin=1.0in]{geometry}

\newcommand{\nnn}{\ensuremath{{n\in{\mathbb N}}}}

\newcommand{\menge}[2]{\big\{{#1}~\big |~{#2}\big\}}

\newcommand{\fenv}[1]%
{\ensuremath{\,\overrightarrow{\operatorname{env}}_{#1}}}
\newcommand{\benv}[1]%
{\ensuremath{\,\overleftarrow{\operatorname{env}}_{#1}}}

\newcommand{\RR}{\ensuremath{\mathbb R}}

\def\la{\langle}

\newcommand{\CC}{\ensuremath{\mathbb C}}
\newcommand{\NN}{\ensuremath{\mathbb N}}

\newcommand{\lm}{\lambda}

\newcommand{\ran}{\ensuremath{\operatorname{ran}}}

\newcommand{\Fix}{\ensuremath{\operatorname{Fix}}}

\newcommand{\Id}{\ensuremath{\operatorname{Id}}}

\def\disp{\displaystyle}
\def\ve{\varepsilon}

\def\gg{\gamma}

\def\ra{\rangle}
\def\la{\langle}

\def\mcX{\mathcal{X}}
{\begin{list}{}{%
\settowidth{\labelwidth}{\textrm{#1~}}%
\setlength{\leftmargin}{\labelwidth+\labelsep}}}%requires macro calc.sty
{\end{list}}
\newtheorem{theorem}{Theorem}[section]
\newtheorem{lemma}[theorem]{Lemma}
\newtheorem{corollary}[theorem]{Corollary}
\newtheorem{proposition}[theorem]{Proposition}
\newtheorem{definition}[theorem]{Definition}
\theoremstyle{plain}{\theorembodyfont{\rmfamily}
}
\theoremstyle{plain}{\theorembodyfont{\rmfamily}
}
\theoremstyle{plain}{\theorembodyfont{\rmfamily}
}
\theoremstyle{plain}{\theorembodyfont{\rmfamily}
\newtheorem{example}[theorem]{Example}}
\newtheorem{fact}[theorem]{Fact}
\theoremstyle{plain}{\theorembodyfont{\rmfamily}
\newtheorem{remark}[theorem]{Remark}}

\def\endproof{\ensuremath{\hfill \quad \blacksquare}}

\def\doi{DOI}

\newcounter{count}

\allowdisplaybreaks

\begin{document}

\title{\textrm{Optimal rates of convergence of
matrices with applications}}

\author{
Heinz H.\ Bauschke\thanks{Mathematics, University of British
Columbia, Kelowna, B.C.\ V1V~1V7, Canada. E-mail:
\texttt{heinz.bauschke@ubc.ca}.},~
J.Y.\ Bello Cruz\thanks{IME, Federal University of Goias,
Goiania, G.O. 74001-970, Brazil. E-mail:
\texttt{yunier.bello@gmail.com}.},~Tran T.A.\ Nghia\thanks{Mathematics, University of British Columbia, Kelowna, B.C.\ V1V~1V7, Canada. E-mail: \texttt{nghia.tran@ubc.ca}.},~
Hung M.\ Phan\thanks{Mathematics, University of British Columbia, Kelowna, B.C.\ V1V~1V7, Canada. E-mail:  \texttt{hung.phan@ubc.ca}.}, ~and
Xianfu\ Wang\thanks{Mathematics, University of British Columbia,
Kelowna, B.C.\ V1V~1V7, Canada. E-mail:
\texttt{shawn.wang@ubc.ca}.}}

\date{\today}

\maketitle \thispagestyle{fancy}

\vskip 8mm

\begin{abstract} \noindent
We present a systematic study  on the linear convergence rates of
the powers of (real or complex) matrices. We derive a characterization
when the optimal convergence rate is attained. This
characterization is given in terms of semi-simpleness of all
eigenvalues having the second-largest modulus after 1.
We also provide applications of our general results to analyze
the optimal convergence rates for several relaxed alternating projection
 methods and the generalized Douglas-Rachford splitting methods for
 finding the projection on the intersection of two subspaces. Numerical experiments
 confirm our convergence analysis.
\end{abstract}

{\small \noindent {\bfseries 2010 Mathematics Subject
Classification:} {Primary 65F10, 65F15;
Secondary 65B05, 15A18
% 26B25 convexity generalizations? no
% 47H10 fixed point theorems? no
}
}


%%%%%%%%%%%%%%%%%%%%%%%%%%%%%%%%%%%%%%%%%%%%%%%%%%%%%%%%%%%%%%%%%%%%%%%%%%%%%%%%%%%

\section{Introduction}
The focus of this paper is the study of the convergence rate of the powers
of a real or complex matrix $A$.  Necessary and sufficient conditions
for such convergence rates were first established by Hensel \cite{Hen}
and later by Oldenburger \cite{Old}. The
convergence rate plays a central role in many well-known
algorithms for solving linear systems such as  \emph{Jacobi,
Gauss-Seidel, successive over-relaxation} methods; see, e.g.,
\cite{Me-Pl,Saad}. Furthermore,  the convergence of the power $A^k$
is linear and the rate is \emph{dominated} by the second-largest
absolute eigenvalue of $A$, $\gg(A)$, which relates to the
\emph{subdominant or controlling eigenvalue} \cite{K,NN}. Natural
questions thus arising are ``What is the optimal (smallest) convergence
rate?'' and ``When is $\gg(A)$ the
optimal convergence rate?''. In general, the optimal convergence rate does not
exist (see Example~\ref{ex:lm} below). However, many iterative
linear methods such as  the {\em method of alternating
projections} (also known as von Neumann's method) \cite{bb96,Maratea} and the
\emph{Douglas-Rachford splitting algorithm} \cite{DR,EB,HLN,LM} do
obtain the optimal linear rates of convergence; see also \cite{BCNPW,DZ,KW}.
We are thus investigating in which case the convergence of the powers
$A^k$ admits the optimal linear rate. We will provide  complete
answers for aforementioned questions in Theorem~\ref{t:linearII}
and Theorem~\ref{t:main}. Furthermore, we then are in a position
to analyze convergence rates of relaxed
alternating projection and generalized
Doughlas-Rachford algorithms for subspaces.

The rest of the paper is organized as follows. In Section~2 we
systematically study convergence rates of matrices.
The main result in this section is Theorem~\ref{t:main}, which gives
a necessary and sufficient condition for the powers $A^k$ to converge
linearly  with the optimal rate $\gg(A)$ via the semi-simpleness of
all the eigenvalues having the second-largest absolute values among
the spectrum. Section~3 is devoted to the applications of Section~2
to the relaxed alternating methods and also the generalized
Douglas-Rachford splitting methods.
In Section~4 we
introduce and study a nonlinear map that helps to accelerate the
convergence of the alternating projection method.
In Section~\ref{s:numerical}, we present some numerical results to illustrate our convergence
theory developed in earlier sections. Finally, we present our conclusions in Section~\ref{s:conclusion}.
 %a summary of main results and possible
%future applications.

{\bf Notation.}  Throughout, we denote by $\CC^{n\times n}$ and
$\RR^{n\times n}$ the sets of $n\times n$ complex matrices and real
matrices, respectively. Let $A$ be a matrix in $\CC^{n\times n}$
(or $\RR^{n\times n}$). The notation $A^*$ stands for the adjoint
(complex transposed)
matrix of $A$.  The matrix norm used in this paper is the \emph{operator
norm}, i.e., $\|A\|=\max\{\|Ax\||\; x\in \CC^n, \|x\|\le 1\}$. We
write $\ker A$, $\ran A$, and ${\rm rank}\, A$ as the kernel,  range,
rank of $A$, respectively. Moreover, $\Fix A:=\ker(A-\Id)$ is known
as the set of fixed points of $A$, where $\Id$ is the identity
mapping. We say $A$ is \emph{nonexpansive} if $\|Ax\|\le \|x\|$ for
all $x\in \CC^{n}$; furthermore, $A$ is \emph{firmly nonexpansive}
if  $\|Ax\|^2+\|x-Ax\|^2\le \|x\|^2$ for all $x\in \CC^{n}$. For
any subspace $U$ of $\RR^n$, the notation $P_U$ is referred to the
orthogonal \emph{projection operator} to $U$, $\dim U$ for the
dimension of $U$, and $U^{\perp}$ for the orthogonal complement of
$U$. We denote $I_n, 0_n, 0_{m\times n}$ by the $n\times n$ identity
matrix, the $n\times n$ zero matrix, and the $m\times n$ zero matrix,
respectively.

\section{The optimal convergence rate of matrices}
In this section we establish conditions under which convergent matrices attain their optimal convergent rate.
Let us recall some definitions and facts used in the sequel.

\begin{definition}[convergent matrices] \label{d:CM}Let $A\in
\CC^{n\times n}$. We say $A$ is \emph{convergent}\footnote{{In
the literature, $A$ is called convergent if the power $A^k$ converges to $0$; moreover, $A$ is semi-convergent whenever the latter limit $A^k$ exists. To avoid the confusion of these two terminologies, we just say $A$ is convergent in both cases.}} to $A^\infty\in \CC^{n\times n}$ if and only if
\begin{eqnarray}\label{CO}
\|A^k-A^\infty \|\to 0\quad \mbox{as}\quad k\to \infty.
\end{eqnarray}
We say $A$ is \emph{linearly} convergent to $A^\infty$ with rate $\mu\in [0,1)$ if there are some $M, N>0$ such that
\begin{eqnarray}\label{Rate}
\|A^k-A^\infty \|\le M\mu^k\quad \mbox{for all}\quad  k>N, k\in \NN.
\end{eqnarray}
Then $\mu$ is called a \emph{convergence rate} of $A$. When the infimum
of all the convergence rates is also a convergence rate, we say
this minimum is the \emph{optimal convergence rate}.
\end{definition}

For any $A\in \CC^{n\times n}$ we denote by $\sigma(A)$ the \emph{spectrum} of $A$, the set of all eigenvalues. The spectral radius \cite[Example~7.1.4]{Meyer} of $A$ is defined by
\begin{eqnarray}\label{e:radi}
\rho(A):=\max\{|\lm||\; \lm\in \sigma(A)\}.
\end{eqnarray}
The next fact is the classical formula of spectral radius.
 \begin{fact}[spectral radius formula]\emph{(\cite[Example~7.10.1]{Meyer})}\label{spec-for} Let $A\in \CC^{n\times n}$. Then we have
\begin{eqnarray}\label{SRF}
\rho(A)=\lim_{k\to\infty}\|A^k\|^\frac{1}{k}.
\end{eqnarray}
\end{fact}

With $\lm\in \sigma(A)$, recall from \cite[page 587]{Meyer} that  ${\rm index}\, (\lm)$ is the smallest positive integer $k$ satisfying ${\rm rank}\, (A-\lm \Id)^k={\rm rank}\, (A-\lm \Id)^{k+1}$. Furthermore, we say $\lm\in \sigma(A)$ is {\em semisimple} if ${\rm index}\, (\lm)=1$; see, e.g., \cite[Exercise~7.8.4]{Meyer}.

\begin{fact}\label{f:ss} For $A\in \CC^{n\times n}$, $\lm\in \sigma(A)$ is  semisimple if and only if $\ker (A-\lm \Id)=\ker (A-\lm \Id)^2$.
\end{fact}
\begin{proof} Note that $\lm\in \sigma(A)$ is  semisimple if and only if
\[
\dim[\ker (A-\lm \Id)]=n-{\rm rank}\, (A-\lm \Id)=n-{\rm rank}\, (A-\lm \Id)^{2}=\dim [\ker (A-\lm \Id)^2].
\]
Since $\ker (A-\lm \Id)\subset \ker (A-\lm \Id)^2$, the equality $\dim[\ker (A-\lm \Id)]=\dim [\ker (A-\lm \Id)^2]$ holds if and only if $\ker (A-\lm \Id)= \ker (A-\lm \Id)^2$. This verifies the proof of the fact.
\end{proof}

The following result taken from \cite{Meyer} gives us a complete characterization of a convergent matrix.

\begin{fact}[limits of powers]\emph{(\cite[page 617-618 and page
630]{Meyer})}\label{t:limitp} For $A\in \CC^{n\times n}$, $A$ is
convergent to $A^\infty$ if and only if
\begin{eqnarray}\label{spec-con}
\label{less1}&&  \rho(A)<1,\; \text{or else}\\
\label{equal1}&&  \rho(A)=1 \;\mbox{ and }\; \lm=1\; \mbox{is semisimple and it is the only eigenvalue on the unit circle}.
\end{eqnarray}
When this happens, we have
\begin{equation}\label{Q}
\disp A^\infty=\text{ the projector onto $\ker(A-\Id)$ along $\ran (A-\Id)$}.
\end{equation}
In particular, when $\rho(A)<1$, we have $ A^\infty=0$.
\end{fact}
The proof of the above fact is indeed based on the spectral resolution of $A^k$ stated below.
\begin{fact}[spectral resolution of $A^k$]\emph{(\cite[page 603 and page 629]{Meyer})}\label{f:spectralres}
For $k\in \NN$ and $A\in \CC^{n\times n}$ with $\sigma(A)=\{\lambda_{1},\lambda_{2},\ldots, \lambda_{s}\}$
and  $k_{i}={\rm index}\, (\lambda_{i})$,  we have
\begin{equation}\label{e:spec}
A^{k}=\sum_{i=1}^{s}\lambda_{i}^{k}G_{i}+\sum_{i=1}^{s}\sum_{j=1}^{k_{i}-1}
\begin{pmatrix}
k\\
j
\end{pmatrix}
\lambda_{i}^{k-j}(A-\lambda_{i}\Id)^{j}G_{i},
\end{equation}
where the spectral projector $G_{i}$'s have the following properties:
\begin{enumerate}
\item $G_{i}$ is the projector onto  $\ker((A-\lambda_{i}\Id)^{k_{i}})$
along $\ran ((A-\lambda_{i}\Id)^{k_{i}})$.
\item $G_{1}+G_{2}+\cdots +G_{s}=\Id$.
\item $G_{i}G_{j}=0$ when $i\neq j$.
\item $N_{i}=(A-\lambda_{i}\Id)G_{i}=G_{i}(A-\lambda_{i}\Id)$ is \emph{nilpotent} of index
$k_{i}$, i.e., $N_i^{k_i}=0$ and $N_i^{k_i-1}\neq0$ .
\end{enumerate}
\end{fact}

\begin{remark}\label{r:ma} Note from Fact~\ref{f:spectralres} (i) and (iv) that
\[
0\neq N_i^{k_i-1}=(A-\lambda_{i}\Id)^{k_i-1}G_{i}^{k_i-1}=(A-\lambda_{i}\Id)^{k_i-1}G_{i}\quad \mbox{if}\quad k_i>1.
\]

\end{remark}

\begin{corollary} \label{c:non} {Suppose that  $A\in \CC^{n\times n}$ is  convergent to $A^\infty\in \CC^{n\times n}$. Then the following hold:}

{{\bf (i)} $A^\infty=P_{\Fix A}$  if and only if $\Fix A=\Fix A^*$.}

{{\bf (ii)} If $A$ is nonexpansive or normal, then $A^\infty=P_{\Fix A}$.}
\end{corollary}
\begin{proof} {It follows from \eqref{Q} that $A^\infty$ is equal to the projector onto $\ker(A-\Id)$ along $\ran (A-\Id)$. Thanks to the equality \cite[(5.9.11)]{Meyer}, we have  $\ran(A^\infty-\Id)=\ran(A-\Id)$. If  $A^\infty=P_{\Fix A}$, we obtain
\begin{equation*}
\ran(A-\Id)=\ran(A^\infty-\Id)=\ran (P_{\Fix A}-\Id)=\ran(P_{(\Fix A)^\perp})=(\Fix A)^\perp.
\end{equation*}
It follows that
\begin{equation*}
\Fix A=\big[(\Fix A)^\perp\big]^\perp=\ran(A-\Id)^\perp=\ker(A^*-\Id)=\Fix A^*.
\end{equation*}
Conversely, if $\Fix A=\Fix A^*$, we have
\begin{equation*}
\ker(A-\Id)=\Fix A=\Fix A^*=\ker(A^*-\Id)=\ran(A-\Id)^\perp,
\end{equation*}
which implies in turn that  the projector onto $\ker(A-\Id)$ along $\ran (A-\Id)$ is exactly the orthogonal projection $P_{\Fix A}$. The first part {\bf (i)} of the corollary is complete.}

{To justify the second part {\bf (ii)}, suppose in addition that $A$ is nonexpansive. Then $\Fix A=\Fix A^*$ by \cite[Lemma~2.1]{BDHP} and thus $A$ is convergent to $P_{\Fix A}$. Moreover, if $A$ is normal, then $A-\Id$ is also normal. Hence for all $x\in \CC^n$ we have
\begin{equation*}
\|(A-\Id)x\|^2=\la (A-\Id)^*(A-\Id)x,x\ra=\la (A-\Id)(A-\Id)^*x,x\ra= \|(A-\Id)^*x\|^2.
\end{equation*}
The latter clearly shows that $\Fix A=\Fix A^*$ and thus $A^\infty=P_{\Fix A}$. The proof is complete.} \end{proof}

%\begin{remark}[what nonexpansiveness buys]
%Now assume we are in the convergent case and $A$ is nonexpansive.
%Then $\Id-A$ is (maximally) monotone (by \cite[Example~20.26]{BC2011}), which implies that
%$\ran(\Id-A)= \clran(\Id-A) = \clran(\Id-A^*) = (\ker(\Id-A))^\perp = (\Fix
%A)^\perp$ (by \cite[Proposition~20.17]{BC2011}),
%and so the oblique projector is the orthogonal one!
%In fact, the argument only requires that $A$ is ``pseudononexpansive'',
%i.e. (in the linear case)
%\begin{equation}
%(\forall x\in X)\quad
%\|Ax\|^2 \leq \|x\|^2 + \|x-Ax\|^2,
%\end{equation}
%see \cite[Example~20.8]{BC2011}, i.e.,
%that
%\begin{equation}
%A+A^*\preceq 2\Id.
%\end{equation}
%\end{remark}

\begin{remark}[convergence, firmly nonexpansiveness and nonexpansiveness]
Let $A\in \RR^{n\times n}$.  When $A$ is firmly nonexpansive, $A$ is convergent; see, e.g., \cite[Example~5.17]{BC2011}.
However, the converse implication fails. Indeed,  consider, for $n\geq 2$,
\begin{equation*}
A =
\begin{pmatrix}
0 & n^{-2}\\
n & 0
\end{pmatrix}.
\end{equation*}
Then $A$ is not (firmly) nonexpansive  because $Ae_1 = ne_2$ where $e_{1}=(1,0)^{\intercal}$ and 
$e_{2}=(0,1)^{\intercal}$.
On the other hand, the characteristic polynomial is
$\lambda\mapsto \lambda^2 - n^{-1}$, which has roots
$\pm n^{-1/2}$. Thus $A$ is convergent due to Fact~\ref{t:limitp}. Moreover, convergence and nonexpansiveness are independent, e.g., $A=-\Id$ is nonexpansive but not convergent.
\end{remark}

We will prove later in this section that whenever $A$ is convergent to $A^\infty$, it is linearly convergent with the  rate not smaller than $\rho(A-A^\infty)$. To manipulate this idea, let us take into account the case of diagonalizable matrices as follows.
\begin{example}[diagonalizable case]\label{diag}
Suppose that $A\in \CC^{n\times n}$ is \emph{diagonalizable} and that $\sigma(A)=\{\lm_1,\ldots,\lm_s\}$ with
\begin{equation*}
1=\lm_1>|\lambda_{2}|\geq |\lambda_{3}|\geq \cdots\geq |\lambda_{s}|.
\end{equation*}
By Fact~\ref{f:spectralres} and Fact~\ref{t:limitp}, we have $A$ is convergent to $A^\infty$ and that
\begin{align*}
A^{k}=A^\infty+\lambda_{2}^kG_{2}+\cdots+\lambda_{s}^{k}G_{s},
\end{align*}
which yields
\begin{align*}
A^{k}-A^\infty=\lambda_{2}^kG_{2}+\cdots+\lambda_{s}^{k}G_{s}.
\end{align*}
It follows that
\begin{subequations}
\begin{align*}
\|A^{k}-A^\infty\|& \leq |\lambda_{2}|^{k}\bigg[ \Big(\frac{|\lambda_{2}|}{|\lambda_{2}|}\Big)^{k}\|G_{2}\|+\cdots +
\Big(\frac{|\lambda_{s}|}{|\lambda_{2}|}\Big)^{k}\|G_{s}\|\bigg]\\
&\le |\lambda_{2}|^{k}\big( \|G_{2}\|+\cdots +
\|G_{s}\|\big).
\end{align*}
\end{subequations}
Hence $A^{k}\rightarrow A^\infty$ with  the linear rate $|\lambda_{2}|$.
\end{example}

In general an eigenvalue having second-largest modulus after 1 is called a subdominant eigenvalue.
\begin{definition}[subdominant eigenvalues] \emph{(\cite{K,NN})} For $A\in \CC^{n\times n}$, we define
\begin{align}\label{subdo}
\gamma(A):=\max\big\{|\lm| |\;\lm\in\{0\}\cup\sigma(A)\setminus\{1\}\big\}.
\end{align}
An eigenvalue $\lm\in \sigma(A)$ satisfying $|\lm|=\gg(A)$ is referred as a \emph{subdominant} eigenvalue.
\end{definition}
When $A$ is not diagonalizable, $\gg(A)$ need not be the convergence rate.
\begin{example}\label{ex:lm} Let us consider the following matrix
\begin{align*}
A=\begin{pmatrix}
1 & 0 & 0\\
0& \frac{1}{2} &  1 \\
0  & 0 & \frac{1}{2}
\end{pmatrix},
\end{align*}
which gives us that  $\gg(A)=\frac{1}{2}$. Note also that $A$ is not  diagonalizable. Moreover,  by induction it is easy to check that
\begin{align*}
A^k=\begin{pmatrix}
1 & 0 & 0\\
0& \frac{1}{2^k} &   \frac{k}{2^{k-1}} \\
0  & 0 & \frac{1}{2^{k}}
\end{pmatrix}\qquad\mbox{for all}\quad k\in \NN.
\end{align*}
Hence we have  $A^k\to A^\infty:=\begin{pmatrix}
1 & 0 & 0\\
0& 0 &  0 \\
0  & 0 & 0
\end{pmatrix}$ as $k\to \infty$. However, observe that
\begin{align*}
\frac{\|A^k-A^\infty\|}{\gg(A)^k}=2^k\left\|\begin{pmatrix}
0 & 0 & 0\\
0& \frac{1}{2^k} &   \frac{k}{2^{k-1}} \\
0  & 0 & \frac{1}{2^{k}}
\end{pmatrix}\right\|=\left\|\begin{pmatrix}
0 & 0 & 0\\
0& 1 &   2k \\
0  & 0 & 1
\end{pmatrix}\right\|\to \infty \quad\mbox{as}\quad k\to \infty.
\end{align*}
Hence $\gg(A)$ is not a convergence rate. However, observe further that any $\mu\in(\frac{1}{2},1)$ is a convergence rate of $A$. Thus  $A$ does not obtain the optimal convergence rate.    \endproof
\end{example}

Our first main result below shows that whenever a matrix $A$ is
convergent, it must be linearly convergent with any rate in
$(\gg(A),1)$. {The theorem can be extended for
linear operator in infinite-dimensional spaces by connecting the
proof below with those of \cite[Theorem~2.1 and 2.2]{BGM}.}

\begin{theorem}\label{t:linear}{\bf (rate of convergence I)}
Suppose that  $A\in \CC^{n\times n}$ is convergent to $A^\infty\in \CC^{n\times n}$. Then we have $\gamma(A)=\rho(A-A^\infty)<1$ and that
\begin{align}\label{e:power}
(A-A^\infty)^k=A^k-A^\infty\quad \mbox{for all}\quad k\in \NN.
\end{align}
Moreover, the following two assertions are satisfied:

{\bf (i)} $A$ is linearly convergent with any rate $\mu \in ( \gamma(A),1)$.

{\bf (ii)} If  $A$ is linearly convergent with rate $\mu\in [0,1)$, then $\mu \in [\gamma(A),1)$.
\end{theorem}
\begin{proof} First let us justify that $\gamma(A)=\rho(A-A^\infty)<1$ and \eqref{e:power} by considering the two following cases taken from  Fact~\ref{t:limitp}:

{\em Case 1.}  $ \rho(A)<1$. In this case we have $A^\infty=0$ by \eqref{SRF}. It follows that $\gamma(A)=\rho(A)=\rho(A-A^\infty)<1$. Note also that \eqref{e:power} is trivial, since $A^\infty=0$.

{\em Case 2.} $\rho(A)=1$ and  $\lm=1$ is semisimple and the only eigenvalue on the unit circle. Suppose that $\sigma(A)\setminus\{1\}=\{\lm_2,\ldots,\lm_s\}$ with $1>|\lm_2|\ge \ldots\ge|\lm_s|$. The Jordan decomposition \cite[page 590]{Meyer} of $A$ allows us to find an invertible matrix $P\in\CC^{n \times n}$ and $r>0$ such that
\begin{equation}\label{Jordan}
A=PJP^{-1}
\end{equation}
with $J$ being the Jordan form of $A$,
\begin{align*}J=\begin{pmatrix}
I_{r} & 0 & \cdots & 0\\
0 & J(\lambda_{2}) & \cdots & 0\\
\vdots & \vdots & \ddots & \vdots\\
0 & 0 & \cdots & J(\lambda_{s})
\end{pmatrix}, \qquad \qquad
J(\lambda_{j})=\begin{pmatrix}
J_{1}(\lambda_{j}) & 0 & \cdots & 0\\
 0 & J_{2}(\lambda_{j}) &  \ddots & 0\\
 \vdots & \vdots & \ddots & \vdots \\
0 &  0 & \cdots & J_{t_{j}}(\lambda_{j})
\end{pmatrix},
\end{align*}
and
\begin{align*}J_{*}(\lambda_{j})=\begin{pmatrix}
\lambda_{j} & 1 & &\\
&  \ddots & \ddots & \\
&   & \ddots & 1\\
& & & \lambda_{j}
\end{pmatrix}.
\end{align*}
Moreover, it follows from \cite[p. 629]{Meyer} that
\begin{equation}\label{e:infy}
A^\infty=P\begin{pmatrix}
I_{r} & 0 \\
0 & 0
\end{pmatrix} P^{-1}.
\end{equation}
This together with the Jordan decomposition above gives us that
\begin{equation}\label{e:dif}
A-A^\infty=P\begin{pmatrix}
0_{r} & 0 & \cdots & 0\\
0 & J(\lambda_{2}) & \cdots & 0\\
\vdots & \vdots & \ddots & \vdots\\
0 & 0 & \cdots & J(\lambda_{s})
\end{pmatrix}P^{-1},
\end{equation}
which readily yields $\rho(A-A^\infty)=\max\{0,|\lm_2|\}=\gg(A)<1$. Observe further from \eqref{Jordan} and \eqref{e:infy} that $AA^\infty=A^\infty A=(A^\infty)^2=A^\infty$. For any $k\in \NN$ the latter gives us that
\begin{align*}
 (A^k-A^\infty)(A-A^\infty)=A^{k+1}-A^kA^\infty-A^\infty A+(A^\infty)^2=A^{k+1}-A^\infty-A^\infty +A^\infty=A^{k+1}-A^\infty.
 \end{align*}
By using this expression, we may prove by induction \eqref{e:power} and thus completes the first part of the theorem.

  Now to verify (i), pick any $\mu\in (\gg(A),1)=(\rho(A-A^\infty),1)$.  Employing \eqref{SRF} for operator $A-A^\infty$ allows us to find some $N\in \NN$ such that
\begin{eqnarray*}
\|A^k-A^\infty\|=\|(A-A^\infty)^k\|\le \mu^k\quad \mbox{for all}\quad k\ge N,
\end{eqnarray*}
which verifies the linear convergence of $A$ with rate $\mu$.\\[-2ex]

It remains to prove (ii). Suppose that $A$ is convergent to $A^\infty$ with rate $\mu\in [0,1)$. Hence there are some $M,N>0$ such that
\[
\|A^k-A^\infty\|\le M\mu^k\quad \mbox{for all}\quad k>N, k\in \NN.
\]
Combining this with the spectral radius formula \eqref{SRF} and \eqref{e:power} gives us that
\[
 \gg(A)= \rho(A-A^\infty)=\lim_{k\to\infty}\|(A-A^\infty)^k\|^\frac{1}{k}=\lim_{k\to\infty}\|A^k-A^\infty\|^\frac{1}{k}\le \lim_{k\to\infty}M^\frac{1}{k}\mu=\mu,
\]
which ensures $ \gg(A)\le \mu$ and thus completes the proof of the theorem.
\end{proof}

A natural question arising from the above theorem is that  in which case $ \gg(A)$ is the optimal convergence rate of $A$; see our Definition~\ref{d:CM}. By Theorem~\ref{t:linear}, the actual problem is that when  $\gg(A)$ is a convergence rate of $A$; see also our Example~\ref{ex:lm}. The next theorem gives us a complete answer for this question.

\begin{theorem}[rate of convergence II] \label{t:linearII}
Let $A\in \CC^{n\times n}$  be convergent to $A^\infty\in\CC^{n\times n}$. Then $\gg(A)$ is the optimal convergence rate of  $A$ if and only if all the subdominant eigenvalues are semisimple.
\end{theorem}
\begin{proof} {By Theorem~\ref{t:linear}, we only need to prove that $\gg(A)$ is a convergence rate of  $A$ if and only if $\lm$ is semisimple for every eigenvalue $\lm\in \sigma(A)$ satisfying $|\lm|= \gg(A)$. For the matrix $A$, denote the set of
distinct eigenvalues in $\sigma(A)\setminus\{1\}$ by $\{\lambda_{2}, \ldots, \lambda_{s}\}$ and
$k_{i}=\text{index}(\lambda_{i})$, $i=2,\ldots,s$. This set may be empty, but in this case we have $\{1\}=\sigma(A)$ and $\text{index}(1)=1$ by Fact~\ref{t:limitp}; thus $A^k=G_1=A^\infty$ by \eqref{Q} and \eqref{e:spec} for all $k\in \NN$, which ensures that $\gg(A)=0$ is a convergence rate of $A$. From now on we suppose that $\sigma(A)\setminus\{1\}\not=\emptyset$ and that
\[
1>|\lambda_{2}|\geq |\lambda_{3}|\geq\cdots\geq |\lambda_{s}|.
\]
If $1\notin \sigma(A)$, we get from Fact~\ref{t:limitp} that $A^\infty=0$ and from \eqref{e:spec} that
\begin{align}\label{spec2a}
A^k=\sum_{i=2}^s\lm_i^kG_i+\sum_{i=2}^{s}\sum_{j=1}^{k_{i}-1}\begin{pmatrix}
k \\
j
\end{pmatrix}\lambda_i^{k-j}(A-\lambda_{i}\Id)^{j}G_{i}=A^\infty+\sum_{i=2}^{s}\sum_{j=0}^{k_{i}-1}\begin{pmatrix}
k \\
j
\end{pmatrix}\lambda_i^{k-j}(A-\lambda_{i}\Id)^{j}G_{i}.
\end{align}
If $\lm_1:=1\in \sigma(A)$, Fact~\ref{t:limitp} tells us that its index  is $1$. Hence, we obtain from  \eqref{Q} that $A^\infty=G_1$. This together with the spectral resolution \eqref{e:spec} gives us that
\begin{equation}\label{spec2b}
A^{k}=G_1+\sum_{i=2}^s\lm_i^kG_i+\sum_{i=2}^{s}\sum_{j=1}^{k_{i}-1}\begin{pmatrix}
k \\
j
\end{pmatrix}\lambda_i^{k-j}(A-\lambda_{i}\Id)^{j}G_{i}=A^\infty+\sum_{i=2}^{s}\sum_{j=0}^{k_{i}-1}\begin{pmatrix}
k \\
j
\end{pmatrix}\lambda_i^{k-j}(A-\lambda_{i}\Id)^{j}G_{i}.
\end{equation}
From both cases above, we always have
\begin{equation}\label{spec2}
A^{k}=A^\infty+\sum_{i=2}^{s}\sum_{j=0}^{k_{i}-1}\begin{pmatrix}
k \\
j
\end{pmatrix}\lambda_i^{k-j}(A-\lambda_{i}\Id)^{j}G_{i}.
\end{equation}}

{If $ \gg(A)=0$, then $\lm_2=0$ and $s=2$,  by \eqref{spec2} we have $A^k=A^\infty$ for all $k\ge 1$. This means that $A=A^\infty$ and $A^2=A$, which ensures that $\lm_2$ is semisimple and $\gg(A)=0$ is a convergence rate. Thus the statement of the theorem is trivial in this case. It remains to prove the theorem when $\gg(A)>0$.  Denote by
\begin{eqnarray}\label{F-max}\begin{array}{ll}
&\disp E:=\{2,\ldots,s\},\qquad F:=\menge{l\in \NN}{|\lambda_{l}|=|\lambda_{2}|,2\leq l\leq s},\\
 &\disp \alpha:=\max\menge{\text{index}(\lambda_{l})}{l\in F}, \quad \mbox{and}
 \quad S:=\menge{i\in F}{\mbox{index}(\lambda_{i})=\alpha}.
\end{array}
\end{eqnarray}
It is clear that $F\supset S\neq\emptyset$ and $\alpha\ge 1$.
By \eqref{spec2} we have
\begin{equation}
A^{k}-A^\infty=\underbrace{\sum_{i\in S}\sum_{j=0}^{k_{i}-1}\begin{pmatrix}
k\\
j
\end{pmatrix}\lambda_i^{k-j}(A-\lambda_{i}\Id)^{j}G_{i}}_{:=H}
+
\underbrace{\sum_{i\in E\setminus S}\sum_{j=0}^{k_{i}-1}\begin{pmatrix}
k\\
j
\end{pmatrix}\lambda_i^{k-j}(A-\lambda_{i}\Id)^{j}G_{i}}_{:=K}.\label{e:S}
\end{equation}
Note that
\begin{eqnarray}\label{e:K1}
\begin{array}{ll}
\disp\frac{K}{|\lm_2|^k\begin{pmatrix}
k\\
\alpha-1\end{pmatrix}}&\disp= \sum_{i\in E\setminus S}\sum_{j=0}^{k_{i}-1}\begin{pmatrix}
k \\
j
\end{pmatrix}
\frac{\lambda_i^{k-j}}{|\lambda_{2}|^{k}} \frac{1}{{\begin{pmatrix}
k\\
\alpha-1
\end{pmatrix}}}(A-\lambda_{i}\Id)^{j}G_{i}\\
& \disp=
\sum_{i\in (E\setminus S)\cup F}\sum_{j=0}^{k_{i}-1}\begin{pmatrix}
k \\
j
\end{pmatrix}
\frac{\lambda_i^{k-j}}{|\lambda_{2}|^{k}} \frac{1}{{\begin{pmatrix}
k\\
\alpha-1
\end{pmatrix}}}(A-\lambda_{i}\Id)^{j}G_{i}\\
&\disp+\sum_{i\in E\setminus F}\sum_{j=0}^{k_{i}-1}\begin{pmatrix}
k \\
j
\end{pmatrix}
\frac{\lambda_i^{k-j}}{|\lambda_{2}|^{k}} \frac{1}{{\begin{pmatrix}
k\\
\alpha-1
\end{pmatrix}}}(A-\lambda_{i}\Id)^{j}G_{i}.
\end{array}
\end{eqnarray}
For $i\in (E\setminus S)\cup F$ and $0\le j\le k_i-1$, observe from the definition of $\alpha$ in \eqref{F-max}  that $j\le \alpha-2$. It follows that
\begin{equation}\label{e:K5}
\frac{\begin{pmatrix}
k \\
j
\end{pmatrix}}{{\begin{pmatrix}
k\\
\alpha-1
\end{pmatrix}}}\big|\frac{\lm_i^{k-j}}{|\lm_2|^k}\big|=\frac{(\alpha-1)!(k-\alpha+1)!}{j!(k-j)!}\frac{1}{|\lm_2|^j}\le \frac{(\alpha-1)!}{j!(k-\alpha+2)}\frac{1}{|\lm_2|^j}:=\ve_1(k)\rightarrow 0
\end{equation}
as $k\to\infty$. For $i\in E\setminus F$ and $0\le j\le k_i-1$, we have
\begin{equation}\label{e:K4}
\left|\begin{pmatrix}
k \\
j
\end{pmatrix}
\frac{\lambda_i^{k-j}}{|\lambda_{2}|^{k}}\right|
\leq k^{j}\Big(\frac{|\lambda_i|}{|\lambda_{2}|}\Big)^{k-j}|\lm_2|^{-j}:=\ve_2(k)\rightarrow 0,
\end{equation}
since $k^{j}$ is polynomial in $k$ and
$(|\lambda_i|/|\lambda_{2}|)^{k-j}$ is exponential with $|\lambda_i|/|\lambda_{2}|<1$ for $i\notin F$. It follows from  \eqref{e:K1},  \eqref{e:K5}, and \eqref{e:K4}   that
\begin{equation}\label{e:K6}
\frac{\|K\|}{|\lm_2|^k\begin{pmatrix}
k\\
\alpha-1\end{pmatrix}}\le  \sum_{i\in (E\setminus S)\cup F}\sum_{j=0}^{k_{i}-1}\ve_1(k)\|(A-\lm_i)^jG_i\|+\sum_{i\in E\setminus F}\sum_{j=0}^{k_{i}-1}\ve_2(k)\|(A-\lm_i)^jG_i\|\to 0.
\end{equation}}

{Next let us justify the ``$\Rightarrow$'' part by supposing that $A$ is convergent to $A^\infty$  with the rate $ \gg(A)=|\lm_2|\in (0,1)$. Hence there are  some $M,N>0$ such that
\begin{equation}\label{li-ra}
\|A^k-A^\infty\|\le M|\lm_2|^k\quad \mbox{for all}\quad k>N, k\in \NN.
\end{equation}
We will prove that $\alpha=1$. Assume by contradiction that $\alpha>1$ and note from \eqref{e:S} that
\begin{eqnarray}\label{e:H}
\begin{array}{ll}
\disp\frac{H}{|\lm_2|^k\begin{pmatrix}
k\\
\alpha-1\end{pmatrix}}&\disp= \sum_{i\in S}\sum_{j=0}^{\alpha-1}\frac{\begin{pmatrix}
k \\
j
\end{pmatrix}}{{\begin{pmatrix}
k\\
\alpha-1
\end{pmatrix}}}
\frac{\lambda_i^{k-j}}{|\lambda_{2}|^{k}} (A-\lambda_{i}\Id)^{j}G_{i}\\
& \disp=\sum_{i\in S}
\frac{\lambda_i^{k-(\alpha-1)}}{|\lambda_{2}|^{k}} (A-\lambda_{i}\Id)^{\alpha-1}G_{i}
+
\sum_{i\in S}\sum_{j=0}^{\alpha-2}\frac{\begin{pmatrix}
k \\
j
\end{pmatrix}}{{\begin{pmatrix}
k\\
\alpha-1
\end{pmatrix}}}
\frac{\lambda_i^{k-j}}{|\lambda_{2}|^{k}} (A-\lambda_{i}\Id)^{j}G_{i}\\
&\disp
=\sum_{i\in S}
\frac{\lambda_i^{k}}{|\lambda_{2}|^{k}} \lambda_{i}^{-(\alpha-1)}(A-\lambda_{i}\Id)^{\alpha-1}G_{i}
+
\underbrace{\sum_{i\in S}\sum_{j=0}^{\alpha-2}\frac{\begin{pmatrix}
k \\
j
\end{pmatrix}}{{\begin{pmatrix}
k\\
\alpha-1
\end{pmatrix}}}\frac{\lambda_i^{k-j}}{|\lambda_{2}|^{k}} (A-\lambda_{i}\Id)^{j}G_{i}}_{:=H_1}.
\end{array}
\end{eqnarray}
Furthermore, for $i\in S$  and $j\le\alpha-2$ similarly to \eqref{e:K5} we may prove that
\begin{equation*}\label{limit0}
\frac{\begin{pmatrix}
k \\
j
\end{pmatrix}}{{\begin{pmatrix}
k\\
\alpha-1
\end{pmatrix}}}\big|\frac{\lm_i^{k-j}}{|\lm_2|^k}\big|:=\ve_3(k)\rightarrow 0\quad \mbox{when}\quad k\to\infty,
\end{equation*}
which implies in turn that
\begin{align}\label{e:H1}
\|H_1\|\le \sum_{i\in S}\sum_{j=0}^{\alpha-2}\ve_3(k)\|(A-\lambda_{i}\Id)^{j}G_{i}\|\to 0 \quad\mbox{as}\quad k\to \infty.
\end{align}
By dividing \eqref{li-ra} by ${|\lm_2|^k\begin{pmatrix}
k\\
\alpha-1\end{pmatrix}}$ and taking $k\to \infty$, we get from  \eqref{e:S},  \eqref{e:K6}, \eqref{e:H}, and \eqref{e:H1} tells us that
\begin{equation}\label{lim-sup}
\lim_{k\to\infty}\sum_{i\in S}
\frac{\lambda_i^{k}}{|\lambda_{2}|^{k}} \lambda_{i}^{-(\alpha-1)}(A-\lambda_{i}\Id)^{\alpha-1}G_{i}=\lim_{k\to \infty}\frac{M}{\begin{pmatrix}
k\\
\alpha-1
\end{pmatrix}}= 0.\end{equation}
Since  $\big|\frac{\lm_i}{|\lm_i|}\big|=1$ for all $i\in S$, by passing to subsequences  we may assume without loss of generality that for each $i\in S$ the sequence $\Big[\frac{\lm_i}{|\lm_i|}\Big]^k\rightarrow x_i$ with $|x_i|=1$ as $k\to \infty$.  Hence, it follows from \eqref{lim-sup} that
\begin{equation}\label{e:G}
\sum_{i\in S} x_{i}\lambda_{i}^{-(\alpha-1)}(A-\lambda_{i}\Id)^{\alpha-1}G_{i}=0.
\end{equation}
By Fact~\ref{f:spectralres} (i) and (iii), we have $G_{i}G_{j}=0$ when $i\neq j$ and $G_{i}G_{i}=G_{i}$.
For any $l\in S$, multiplying both sides of \eqref{e:G} by $G_{l}$ yields
$$x_{l}\lambda_{l}^{-(\alpha-1)}(A-\lambda_{l}\Id)^{\alpha-1}G_{l}=0,$$
which is impossible since $x_{l}\neq 0$, $|\lambda_{l}|=|\lambda_{2}|\neq 0$,
$(A-\lambda_{l}\Id)^{\alpha-1}G_{l}\neq 0$ by Remark~\ref{r:ma}. Thus, $\alpha=1$, thanks to the definition of $\alpha$ in \eqref{F-max} we get that all $\lm_i$, $i\in F$ has the same index $1$ and complete the first part of the proof.}\\[-2ex]

{Conversely, suppose that  all $\lm_i\in \sigma(A)$ satisfying  $|\lm_i|=\gg(A)=|\lm_2|>0$ are semisimple, which implies $\alpha$ in \eqref{F-max} is $1$. Hence we observe  that from \eqref{e:S} that
\begin{equation}\label{e:G2}
\disp\frac{\|H\|}{|\lm_2|^k}=\big\|\sum_{i\in S}\frac{\lm_i^k}{|\lm_2|^k} G_i
\big\|\le\sum_{i\in S}\|G_i\|.
\end{equation}
Moreover, the term $\frac{\|K\|}{|\lm_2|^k}$ still converges to $0$ as proved in \eqref{e:K6}. Combining this with \eqref{e:G2} and   \eqref{e:S}  gives us that $A$ is convergent to $A^\infty$ with the linear rate $|\lm_2|$. The proof of the theorem is complete.}\end{proof}

\begin{remark} It is worth mentioning that Example~\ref{diag} is also a direct consequence of Theorem~\ref{t:linearII}, since all the eigenvalues of $A$ are semisimple when $A$ is diagonalizable. Moreover, $\gg(A)$ is not the convergence rate in Example~\ref{ex:lm}, since $\frac{1}{2}=\gg(A)$ is not semisimple in this case.
\end{remark}

Next let us summarize Fact~\ref{t:limitp}, Theorem~\ref{t:linear}, and Theorem~\ref{t:linearII} in the following result, which provides a complete characterization for obtaining the optimal convergence rate.

\begin{theorem}[optimal convergence rate]\label{t:main} Let $A\in \CC^{n\times n}$. Then $A$ is convergent with the optimal convergence rate, which is $\gg(A)$ if and only if one of the following holds:

{\bf(i)} $\rho(A)<1$ and all $\lm\in\sigma(A)$ satisfying $|\lm|=\gg(A)$ are semisimple.

{\bf (ii)} $\rho(A)=1$, $\lm=1$ is the only eigenvalue on the unit circle, $\lm=1$ is semisimple, and all $\lm\in\sigma(A)$ satisfying $|\lm|=\gg(A)$ are semisimple.
\end{theorem}
\begin{proof}
{If $A$ is convergent with the optimal convergence rate,  Theorem~\ref{t:linear} tells us that $\gg(A)$ is the optimal convergence rate. Moreover, {\bf (i)} and {\bf (ii)} follow from Fact~\ref{t:limitp} and  Theorem~\ref{t:linearII}. Conversely,  if {\bf (i)} and {\bf (ii)} hold, we also get from Fact~\ref{t:limitp} and  Theorem~\ref{t:linearII} that $A$ is convergent with the optimal rate $\gg(A)$.}
\end{proof}

 \begin{theorem}\label{normalI} Let $A\in\CC^{n\times n}$ be  convergent  to $A^\infty$. Then we have
\begin{eqnarray}\label{i-norm}
 \|A^k-A^\infty\|\le \|A-A^\infty\|^k
 \end{eqnarray}
 and thus $ \gg(A)\le \|A-A^\infty\|$. Furthermore, if $A$ is normal then we have
 \begin{eqnarray}\label{p-norm}
 \|A^k-A^\infty\|=\|A-A^\infty\|^k
 \end{eqnarray}
and $ \gg(A)=\|A-A^\infty\|$ is the optimal convergence rate of $A$.
  \end{theorem}
\begin{proof}
First, observe from \eqref{e:power} in Theorem~\ref{t:linear} that
\[
\|A^k-A^\infty\|=\|(A-A^\infty)^k\|\le \|A-A^\infty\|^k,
\]
which together with \eqref{SRF} for $A-A^\infty$  clearly ensures \eqref{i-norm} and thus $ \gg(A)=\rho(A-A^\infty)\le \|A-A^\infty\|$ by Theorem~\ref{t:linear}.

To justify the second part, suppose that $A$ is convergent and normal.  We claim that $A-A^\infty$ is also normal. This is trivial when $A^\infty=0$. It remains to take into account the case $A^\infty\neq 0$.
Since $A$ is normal, we can find  a diagonal matrix $J={\rm diag}\,(\lm_1,\lm_2,\ldots,\lm_n)$ with $|\lm_1|\ge\ldots\ge|\lm_n|$ and a unitary  matrix $P$ such that $A=PJP^*$. Fact~\ref{t:limitp} tells us that $1\in \sigma(A)$ and $1=\lm_1=\ldots=\lm_r>|\lm_{r+1}|\ge \ldots\ge|\lm_n|$  for some $r\in\NN$. It follows that
\begin{equation}\label{e:Ainf}
A^\infty=\lim_{k\to\infty}A^k=P\begin{pmatrix} I_r & 0\\0 & 0\end{pmatrix}P^*.
\end{equation}
Hence we obtain
\[
A-A^\infty=P\begin{pmatrix}
0_{r\times r} & 0 & \cdots & 0\\
0 & \lm_{r+1} & \cdots & 0\\
\vdots & \vdots & \ddots & \vdots\\
0 & 0 & \cdots & \lm_n
\end{pmatrix}P^{*},
\]
which is a normal matrix. The latter formula together with \eqref{e:power} also gives us that
\[
\|A^k-A^\infty\|=\|(A-A^\infty)^k\|=\|A-A^\infty\|^k=|\lm_{r+1}|^k=[\rho(A-A^\infty)]^k=\gg(A)^k,
\]
which ensures \eqref{p-norm} and completes the proof of the theorem.
\end{proof}

\section{Applications to relaxed alternating projection and generalized Douglas-Rachford methods}
In this section, using results in Section 2 and principal angles between two subspaces, we will analyze convergence rates of relaxed alternating projections and generalized Douglas-Rachford methods for two subspaces comprehensively.  Throughout the section we suppose that $U$ and $V$ are two subspaces of $\RR^n$ with $ 1\le p:=\dim U\le \dim V:=q\le n-1$. Note that the whole section will be not interesting if $\dim U=0$ or $\dim V=n$.  Let us recall the principal angles  and the Friedrichs angles between $U$ and $V$ as follows, which are crucial for our quantitative analysis of convergence rates.

\begin{definition}{\bf (principal angles)}\label{d:PA}\emph{(\cite{BG}, \cite[page 456]{Meyer})} The principal angles $\theta_k\in[0,\frac{\pi}{2}]$, $k=1,\ldots,p$ between $U$ and $V$ are defined by
\begin{eqnarray}\label{pa2}
\begin{array}{ll}
\qquad\cos\theta_k &:=\la u_k,v_k\ra\\
& =\max\Big\{\la u,v\ra\Big|\begin{array}{ll}&u\in U, v\in V, \|u\|=\|v\|=1,\\
& \la u,u_j\ra=\la v,v_j\ra=0, j=1, \ldots, k-1\end{array}\Big\}\quad \mbox{with}\quad u_0=v_0:=0.
\end{array}
\end{eqnarray}
\end{definition}
It is worth mentioning that the vectors $u_k,v_k$ are not uniquely defined, but the principal angles $\theta_k$ are unique with $0\le \theta_1\le\theta_2\le\cdots\le \theta_p\le \frac{\pi}{2}$; see \cite[page 456]{Meyer}).

\begin{definition}{\bf (Friedrichs angle)} \label{d:F}The \emph{cosine of the Friedrichs angle} $\theta_F\in(0,\frac{\pi}{2}]$ between $U$ and $V$ is
\begin{equation}\label{fa}
c_F(U,V):=\max\Big\{\la u,v\ra\big|\;u\in U\cap(U\cap V)^\perp, v\in V\cap(U\cap V)^\perp, \|u\|=\|v\|=1\Big\}.
\end{equation}
\end{definition}
In the following proposition we show that the Friedrichs angle is exactly the $(s+1)$-th principal angle $\theta_{s+1}$ where $s:=\dim(U\cap V)$.
\begin{proposition}{\bf(principal angles and Friedrichs angle)}\label{p:PAF} Let  $s:=\dim (U\cap V)$. Then we have $\theta_k=0$ for $k=1,\ldots,s$ and $\theta_{s+1}=\theta_F>0$.
\end{proposition}
\begin{proof}
Let $x_1,\ldots,x_s$ be an orthonormal basis of the subspace $U\cap V$. We may choose $u_k=v_k=x_k$, $k=1,\ldots,s$ from  \eqref{pa2}. It follows that $\cos\theta_k=\la x_k,x_k\ra=1$  and thus $\theta_k=0$ for all $k=1, \ldots, s$. Moreover, since ${\rm span}\,\{u_1,\ldots, u_s\}={\rm span}\,\{v_1,\ldots,v_s\}=U\cap V$, we obtain from \eqref{pa2} that
\begin{align}
\cos\theta_{s+1} &=\max\big\{\la u,v\ra\big|\;u\in U, v\in V, \|u\|=\|v\|=1, u,v\in (U\cap V)^\perp \big\}.\label{pa3}\end{align}
This together with \eqref{fa} tells us that $\theta_{s+1}=\theta_F$.  The proof is complete.
\end{proof}

The following result follows the idea of \cite{BG,DZ} to construct the orthogonal projections $P_U$ and $P_V$ with the appearance of the principal angles.

\begin{proposition}{\bf(principal angles and orthogonal projections)}\label{p:main} Suppose further that $p+q< n$. Then we may find a orthogonal  matrix $D\in \RR^{n\times n}$ such that
\begin{equation}\label{e:proj}
P_U=D\begin{pmatrix}
I_p &0&0&0\\
0 &0_p &0&0\\
0 &0 & 0_{q-p} &0\\
0 &0 & 0 &0_{n-p-q}
\end{pmatrix} D^*
\qquad\mbox{and}\qquad P_V=D\begin{pmatrix}
C^2 & CS&0&0\\
CS & S^2 & 0&0\\
 0 & 0 & I_{q-p}&0 \\
 0 & 0 & 0&0_{n-p-q}
\end{pmatrix} D^*,
\end{equation}
where $C$ and $S$ are two $p\times p$ diagonal matrices defined by
\begin{eqnarray}\label{e:diag}
C:={\rm diag}\,\big(\cos\theta_1,\ldots,\cos\theta_p\big)\qquad \mbox{and}\qquad S:={\rm diag}\,\big(\sin\theta_1,\ldots,\sin\theta_p\big)
\end{eqnarray}
with the principal angles $\theta_1,\ldots, \theta_p$ between $U$ and $V$ found in {\rm Definition~\ref{d:PA}}. Consequently, we have
\begin{align}\label{e:PUV}
P_UP_V=D\begin{pmatrix}
C^2 &CS&0&0\\
0 &0_p &0&0\\
0 &0 & 0_{q-p}&0\\
0&0 &0 & 0_{n-p-q}
\end{pmatrix} D^*\; \mbox{and}\;  P_{U^\perp}P_{V^\perp}=D\begin{pmatrix}
0_p & 0 &0 &0\\
-CS &C^2 &0&0\\
0& 0 &0_{q-p} &0\\
0 &0 & 0& I_{n-p-q}
\end{pmatrix} D^*.
\end{align}
Furthermore, the orthogonal projection $P_{U\cap V}$ is computed by
\begin{equation}\label{e:UV}
P_{U\cap V}=D\begin{pmatrix}
I_s &0\\
 0 &0_{n-s}
\end{pmatrix} D^*\qquad\mbox{with}\qquad s:=\dim(U\cap V).
\end{equation}

\end{proposition}
\begin{proof} Let $Q_U\in \RR^{n\times p}$, $Q_{U^\perp}\in \RR^{n\times(n-p)}$ and $Q_V\in \RR^{n\times q}$ be three matrices such that their columns form three orthonormal bases for $U, U^\perp$ and $V$, respectively. It follows from \cite[page 430]{Meyer} that $P_U=Q_UQ_U^*$, $I-P_U=P_{U^\perp}=Q_{U^\perp}Q_{U^\perp}^*$ and $P_V=Q_VQ_V^*$. Furthermore, by \cite[Theorem~1]{BG} we have the {\em Singular Valued Decomposition} (SVD) of the $p\times q$ matrix $Q_U^*Q_V$ is
\begin{eqnarray}\label{e:SVD}
Q_U^*Q_V=ACB^* \quad\mbox{with}\quad C={\rm diag}\,(\cos\theta_1,\ldots,\cos\theta_p)\in \RR^{p\times p},
\end{eqnarray}
where  $A\in \RR^{p\times p}$ and $B\in \RR^{q\times p}$  satisfy $AA^*=A^*A=B^*B=I_p$. Since all $p$ columns of $B$ are orthonormal and $p\le q$, we may find a $q\times(q-p)$ matrix $B^\prime$ such that $\widetilde{B}:=(B,B^\prime)\in \RR^{q\times q}$ is orthogonal. Define further that  $D_1:=Q_UA\in \RR^{n\times p}$, we have $D^*_1D_1=A^*Q_U^*Q_UA=A^*A=I_p$. Note further from \eqref{e:SVD} that
\begin{eqnarray}\label{e:SVD1}
P_UQ_V=Q_UQ_U^*Q_V=Q_UAC^*B=D_1CB^*.
\end{eqnarray}
Moreover, we get from \eqref{e:SVD} that
\begin{subequations}
\begin{align*}
[Q^*_{U^\perp}Q_V]^*[Q_{U^\perp}^*Q_V]&=Q_V^*Q_{U^\perp}Q^*_{U^\perp}Q_V=Q_V^*(\Id-P_U)Q_V=\Id-Q_V^*P_UQ_V\\
&=\Id-Q^*_VQ_UQ^*_U
Q_V=\Id-(ACB^*)^*(ACB^*)=\Id-BCA^*ACB^*\\
&=\Id-BC^2B^*=\overline{B}\, \overline {B}^*- \widetilde{B}\begin{pmatrix} C^2 &0\\
0 & 0_{q-p}\end{pmatrix} \widetilde{B}^*= \widetilde{B}\begin{pmatrix} I_p-C^2 &0\\
0 & I_{q-p}\end{pmatrix} \widetilde{B}^*\\
&=\widetilde{B}\begin{pmatrix} S^2 &0\\
0 & I_{q-p}\end{pmatrix} \widetilde{B}^*.
\end{align*}
\end{subequations}
Hence the columns of $\widetilde{B}$ are eigenvectors of $[Q^*_{U^\perp}Q_V]^*[Q_{U^\perp}^*Q_V]$. It follows that the SVD of $Q^*_{U^\perp}Q_V$ has the form
\begin{equation}\label{e:SVD2}
Q^*_{U^\perp}Q_V=A_1\begin{pmatrix}S& 0\\ 0 &I_{q-p} \end{pmatrix}\widetilde{B}^*
\end{equation}
for some $A_1\in \RR^{(n-p)\times q}$ with $A_1^*A_1=I_q$.  Define $D_2:=Q_{U^\perp}A_1\in \RR^{n\times q}$, we have $D_2^*D_2=A_1^*Q^*_{U^\perp}Q_{U^\perp}A_1=A_1^*A_1=I_q$. Moreover, it follows from \eqref{e:SVD2} that
\begin{equation}\label{e:SVD3}
(I-P_U)Q_V=Q_{U^\perp}Q^*_{U^\perp}Q_V=D_2\begin{pmatrix}S& 0\\ 0 &I_{q-p} \end{pmatrix}\widetilde{B}^*.
\end{equation}
Note further that $D_1^*D_2=A^*Q_U^*Q_{U^\perp}A_1=0$, since the columns of $Q_U, Q_{U^\perp}$ are two basis of $U$ and $U^\perp$, respectively. Thus there is  an $n\times (n-p-q)$ matrix $D_3$ such that $D:=(D_1, D_2, D_3)\in \RR^{n\times n}$ is  orthogonal . Combining \eqref{e:SVD1} and \eqref{e:SVD3} gives us that
\[
Q_V=D_1CB^*+D_2\begin{pmatrix}S& 0\\ 0 &I_{q-p} \end{pmatrix}\widetilde{B}^*=D_1\begin{pmatrix}C&0_{p\times (q-p)}\end{pmatrix}\widetilde{B}^*+D_2\begin{pmatrix}S& 0\\ 0 &I_{q-p} \end{pmatrix}\widetilde{B}^*.
\]
Hence we have
\begin{subequations}
\begin{align*}
P_V&=Q_VQ_V^*=\left[D_1\begin{pmatrix}C&0_{p\times (q-p)}\end{pmatrix}\widetilde{B}^*+D_2\begin{pmatrix}S& 0\\ 0 &I_{q-p} \end{pmatrix}\widetilde{B}^*\right]\cdot\left[\widetilde{B}\begin{pmatrix}C\\0_{(q-p)\times p}\end{pmatrix}D_1^*+\widetilde{B}\begin{pmatrix}S& 0\\ 0 &I_{q-p} \end{pmatrix}D_2^*\right]\\
&=D_1C^2D_1^*+D_1\begin{pmatrix}CS&0_{p\times (q-p)}\end{pmatrix}D_2^*+D_2\begin{pmatrix}SC\\ 0_{(q-p)\times p}\end{pmatrix}D^*_1+D_2\begin{pmatrix}S^2&0\\ 0&I_{q-p}\end{pmatrix}D^*_2\\
&=D\begin{pmatrix}
C^2 & CS&0&0\\
CS & S^2 & 0&0\\
 0 & 0 & I_{q-p}&0 \\
 0 & 0 & 0&0_{n-p-q}
\end{pmatrix} D^*,
\end{align*}
\end{subequations}
which ensures the the second part of \eqref{e:proj}. Note further that $D_1D_1^*=Q_UA(Q_UA)^*=Q_UAA^*Q_U^*=Q_UQ_U^*=P_U$. It follows that
\begin{subequations}
\begin{align*}
P_U&=D\begin{pmatrix}
I_p &0&0&0\\
0 &0_p &0&0\\
0 &0 & 0_{q-p} &0\\
0 &0 & 0 &0_{n-p-q}
\end{pmatrix} D^*,
\end{align*}
\end{subequations}
which verifies \eqref{e:proj}. The formulas of $P_UP_V$ and $P_{U^\perp}P_{V^\perp}=(\Id-P_U)(\Id-P_V)$ in \eqref{e:PUV} can be derived easily from \eqref{e:proj}. It remains to establish \eqref{e:UV}.
Observe from \eqref{e:PUV} and Proposition~\ref{p:PAF} that
\begin{align*}
(P_UP_V)^k&=D\begin{pmatrix}
C^{2k} &C^{2(k-1)}CS&0\\
0 &0_p &0\\
0 &0 & 0_{n-2p}
\end{pmatrix} D^*\longrightarrow D\begin{pmatrix}
I_s &0\\
0 &0_{n-s}
\end{pmatrix} D^* \quad \mbox{as}\quad k\to \infty.
\end{align*}
Note further that $\Fix(P_UP_V)=U\cap V=\Fix(P_VP_U)$; see, e.g.,  \cite[Lemma~2.4]{BDHP}. Combining this with \eqref{e:UV} and Corollary~\ref{c:non} tells us that $P_{U\cap V}=P_{\Fix(P_UP_V)}= D\begin{pmatrix}
I_s &0\\
0 &0_{n-s}
\end{pmatrix} D^*$. \end{proof}

\begin{remark} When $p+q<n$, observe from \eqref{e:PUV}, \eqref{fa}, and Proposition~\ref{p:PAF} that $\gg(P_UP_V)=\gg(P_{U^\perp}P_{V^\perp})=c^2_F(U,V)$. These equalities is also true when $p+q\ge n$ by applying the trick used in Case 2 in the proof of Theorem~3.6. It follows that  $c_F(U,V)=c_F(U^\perp,V^\perp)$ by replacing $U,V$ by $U^\perp,V^\perp$, respectively. This equality is known as Solmon's formula; see \cite[Theorem~16]{Maratea} and also \cite[Theorem 3]{miao} for different proofs.
\end{remark}

\subsection{Convergence rate of relaxed alternating projection methods}
Throughout this subsection let us denote the classical \emph{alternating projection} mapping by $T:=P_UP_V$, which is well-known to be convergent to $P_{U\cap V}$ with the linear rate $c^2_F(U,V)=\cos^2\theta_{s+1}$ with $s=\dim(U\cap V)$; see \cite{Maratea,KW}. We will study some relaxations of this operator and show that a better optimal rate can be obtained. We say  the {\em relaxed alternating projection mapping} defined by
\begin{equation}\label{Tmu}
T_\mu:=(1-\mu)\Id+\mu P_UP_V\quad \mbox{with}\quad \mu\in \RR.
\end{equation}
 It is worth noting that the case $\mu=0$ is not interesting, since $T_0=\Id$ is the identity map. Let us analyze the convergence of $T_\mu$ in the following result mainly for the case $\mu\neq 0$. When $\mu=1$, it recovers the classical result aforementioned.

 \begin{theorem}[relaxed alternating projection] \label{t:relaxedI} Let $\theta_{s+1}=\theta_F$ be defined in {\rm Proposition~\ref{p:PAF}} with $s=\dim(U\cap V)$.  Then the mapping $T_\mu=(1-\mu)\Id+\mu P_UP_V$, $\mu\in \RR$ is convergent if and only if $\mu\in [0,2)$. Moreover, the following assertions hold:

  {\bf (i)} If $\mu\in (0,\frac{2}{1+\sin^2 \theta_{s+1}}]$, then $T_\mu$ is convergent to $P_{U\cap V}$ with the optimal rate $\gg(T_\mu)=1-\mu\sin^2\theta_{s+1}$.

{\bf (ii)} If $\mu\in (\frac{2}{1+\sin^2 \theta_{s+1}}, 2)$, then $T_\mu$ is convergent to $P_{U\cap V}$ with the optimal rate $\gg(T_\mu)=\mu-1$.

Consequently, when $\mu\neq 0$, $T_\mu$ is convergent to $P_{U\cap V}$ with rate smaller than $\cos^2\theta_{s+1}$ if and only if $\mu\in (1,2-\sin^2\theta_{s+1})$. Furthermore, $T_\mu$ attains the smallest convergence rate $\frac{1-\sin^2\theta_{s+1}}{1+\sin^2\theta_{s+1}}$ at $\mu= \frac{2}{1+\sin^2 \theta_{s+1}}$.
  \end{theorem}

 \begin{proof} Let us justify the theorem by considering two main cases as follows.

 {\bf Case 1:} $p+q< n$, where $1\le p=\dim U\le q=\dim V\le n-1$.   By Proposition~\ref{p:main},  \eqref{e:proj} and \eqref{e:PUV}, we may find some orthogonal  matrix $D$ such that
 \begin{eqnarray} \label{e:bS}
 \begin{array}{ll}
 T_\mu&\disp=(1-\mu)\Id+\mu P_UP_V\disp=D\begin{pmatrix}
 (1-\mu)I_p+\mu C^2 & \mu CS & 0\\
 0 &(1-\mu) I_p &0\\
 0 & 0 & (1-\mu)I_{n-2p}
 \end{pmatrix}D^* \\
 &\disp=D\begin{pmatrix}
 I_p-\mu S^2 & \mu CS & 0\\
 0 &(1-\mu) I_p &0\\
 0 & 0 & (1-\mu)I_{n-2p}
 \end{pmatrix}D^*. \end{array}
 \end{eqnarray}
It follows that
\begin{equation}\label{e:S10}
\sigma(T_\mu)=\{1-\mu\sin^2\theta_k|\; k=1,\ldots,p\}\cup\{1-\mu\}.
\end{equation}
Suppose first that $T_\mu$ is convergent, we get from Fact~\ref{t:limitp} that  $\rho(T_\mu)\le 1$ and $-1\not\in\sigma(T_\mu)$. Thus we have $|1-\mu|\le 1$ and $-1\neq 1-\mu$, which yield $0\le \mu<2$. Conversely, suppose that  $0\le \mu<2$ and observe from Proposition~\ref{p:PAF} that
\begin{equation*}
1=1-\mu\sin^2\theta_1=\ldots=1-\mu\sin^2\theta_s>1-\mu\sin^2\theta_{s+1}\ge\ldots\ge 1-\mu\sin^2\theta_{p}\ge1-\mu>-1.
\end{equation*}
If $\mu=0$ then $T_\mu=\Id$ is always convergent. If $\mu>0$ and $s=0$, it is clear that $1\notin \sigma(T_\mu)$ by \eqref{e:S10}. Thus  $T_\mu$ is convergent by Fact~\ref{t:limitp}. If $\mu>0$ and $s>0$, we claim that $1\in \sigma(T_\mu)$ is semisimple. Indeed, observe from \eqref{e:bS} that
\begin{equation}\label{n1}
\ker(T_\mu-\Id)=D\begin{pmatrix}\ker (-\mu S^2)\\0_{(n-p)\times 1}\end{pmatrix}= D\begin{pmatrix}
\RR^s\\0_{(n-s)\times 1}\end{pmatrix}.
\end{equation}
Similarly we also have
\begin{equation}\label{n2}
\ker(T_\mu-\Id)^2=D\begin{pmatrix}\ker (-\mu^2 S^4)\\0_{(n-p)\times 1}\end{pmatrix}= D\begin{pmatrix}
\RR^s\\ 0_{(n-s)\times 1}\end{pmatrix}.
\end{equation}
It follows from \eqref{n1} and \eqref{n2} that $1$ is semisimple to $T_\mu$ due to Fact~\ref{f:ss}. This tells that $T_\mu$ is convergent by Fact~\ref{t:limitp}. Thus $T_\mu=(1-\mu)\Id+\mu P_UP_V$, $\mu\in \RR$ is convergent if and only if $\mu\in [0,2)$.

Next let us justify {\bf(i)} and {\bf (ii)} under the assumption that $\mu\in(0,2)$. We claim first that $T_\mu$ is convergent to $P_{U\cap V}$. Indeed, note that
\[
\Fix T_\mu=\ker [\mu(P_UP_V-\Id)]=\ker(P_UP_V-\Id)=\Fix (P_UP_V)=U\cap V.
\]
Furthermore, we  have
\[
\Fix T^*_\mu=\ker [\mu(P_VP_U-\Id)]=\ker(P_VP_U-\Id)=\Fix (P_VP_U)=V\cap U,
\]
which yields in turn the equality $\Fix T_\mu=\Fix T^*_\mu$. By Corollary~\ref{c:non}, the mapping $T_\mu$ is convergent to $P_{U\cap V}$.

Now we justify the quantitative characterizations in {\bf (i)} and {\bf(ii)}.  Observe from \eqref{e:S10} that the subdominant eigenvalue of $T_\mu$ is
\begin{eqnarray}\label{sub2}
\gg(T_\mu)=\max\{|1-\mu\sin^2\theta_{s+1}|,|1-\mu|\}.
\end{eqnarray}
Note also that
\begin{equation}\label{e:m1}
 (1-\mu\sin^2\theta_{s+1})^2-(1-\mu)^2=\mu\cos^2\theta_{s+1}\big[2-\mu(1+\sin^2\theta_{s+1})\big].
\end{equation}

{\em Subcase a:} $\cos^2\theta_{s+1}=0$. Then we have  $\theta_{s+1}=\ldots=\theta_p=\frac{\pi}{2}$ and $\gg(T_\mu)=|1-\mu|$. In this case it is easy to see that $CS=0$ and thus $T_\mu$ is diagonalizable by \eqref{e:bS}. Thanks to Example~\ref{diag} we have $T_\mu$ is convergent with optimal rate $|1-\mu|$. Both {\bf(i)} and {\bf(ii)} are  valid in this case.

{\em Subcase b:} $\cos^2\theta_{s+1}>0$. Let us consider the following three subsubcases:

{\em Subsubcase b1:} $\mu\in (0,\frac{2}{\sin^2 \theta_{s+1}+1})$. Then we have $|1-\mu\sin^2\theta_{s+1}|> |1-\mu|$ by \eqref{e:m1} and thus $\gg(T_\mu)=|1-\mu\sin^2\theta_{s+1}|$. Observe that
\begin{equation}\label{e:m2}
1>a_\mu:=1-\mu\sin^2\theta_{s+1}> 1-\frac{2\sin^2\theta_{s+1}}{1+\sin^2 \theta_{s+1}}=\frac{1-\sin^2\theta_{s+1}}{1+\sin^2 \theta_{s+1}}\ge0.
\end{equation}
Hence we have $\gg(T_\mu)=1-\mu\sin^2\theta_{s+1}$. Suppose further that $\theta_{s+1}=\ldots=\theta_k$ and $\theta_{s+1}\neq\theta_{k+1}$ with some $k\in\{s+1,\ldots,p\}$, we easily check from \eqref{e:bS} that
\begin{equation*}
\ker(T_\mu-a_\mu\Id)=\ker(T_\mu-a_\mu\Id)^2=D\begin{pmatrix}0_{1\times s}\times (\RR^{k-s})^*\times 0_{1\times(n-k)}\end{pmatrix}^*,
\end{equation*}
which shows that $a_\mu$ is semisimple by Fact~\ref{f:ss}. Thanks to Theorem~\ref{t:linearII}, $T_\mu$ is convergent with the optimal rate $a_\mu$.

{\em Subsubcase b2:} $\mu=\frac{2}{1+\sin^2 \theta_{s+1}}>1$. Then we obtain from \eqref{sub2} that
\begin{equation}\label{e:m3}
\gg(T_\mu)=|1-\mu\sin^2\theta_{s+1}|= |1-\mu|=1-\mu\sin^2\theta_{s+1}=\mu-1=\frac{1-\sin^2\theta_{s+1}}{1+\sin^2 \theta_{s+1}}.
\end{equation}
 It is similar to the above subsubcase  that $a_\mu\in \sigma(T_\mu)$ is semisimple. Furthermore, $1-\mu\in \sigma(T_\mu)$ is also semisimple. Indeed, observe that
\[
T_\mu-(1-\mu)\Id=D\begin{pmatrix} \mu C^2 & \mu CS &0\\
0 &0_p &0 \\
0 &0 &0_{n-2p}\end{pmatrix}D^*,\;\;(T_\mu-(1-\mu)\Id)^2=D\begin{pmatrix} \mu^2 C^4 & \mu^2 C^3S &0\\
0 &0_p &0 \\
0 &0 &0_{n-2p}\end{pmatrix}D^*.
\]
By using these two expressions, we may check that
\[
\ker(T_\mu-(1-\mu)\Id)=\ker (T_\mu-(1-\mu)\Id)^2,
\]
which yields that $(1-\mu)$ is also semisimple by Fact~\ref{f:ss}. By Theorem~\ref{t:linearII} again, we obtain that $\frac{1-\sin^2\theta_{s+1}}{1+\sin^2 \theta_{s+1}}$ is the optimal convergent rate of $T_\mu$.

{\em Subsubcase b3:} $\mu>\frac{2}{1+\sin^2 \theta_{s+1}}>1$. It follows from \eqref{e:m1} that  $|1-\mu\sin^2\theta_{s+1}|<|1-\mu|$. And thus we get from \eqref{sub2} that
\begin{equation}\label{e:m4}
\gg(T_\mu)= |1-\mu|=\mu-1>\frac{2}{1+\sin^2 \theta_{s+1}}-1=\frac{1-\sin^2\theta_{s+1}}{1+\sin^2 \theta_{s+1}}.
\end{equation}
Similarly to the above case, $1-\mu\in \sigma(T_\mu)$ is semisimple. Thus Theorem~\ref{t:linearII} tells us that $\mu-1$ is the convergent rate of $T_\mu$ in this subcase.

Combining Subsubcase  b1 and Subsubcase b2  ensures {\bf (i)}, and {\bf (ii)} is exactly the Subsubcase b3. Thus  {\bf (i)} and {\bf (ii)} are verified.

Let us complete the proof by verifying the last part of the theorem. When $\mu\in (0,\frac{2}{1+\sin^2\theta_{s+1}}]$, we have $1-\mu\sin^2\theta_{s+1}<\cos^2\theta_{s+1}$ if and only if $\mu>1$, since $\sin^2\theta_{s+1}>0$ by Proposition~\ref{p:PAF}. Furthermore, when $\mu\in (\frac{2}{1+\sin^2\theta_{s+1}},2)$, we have  $\mu-1<\cos^2\theta_{s+1}$ if and only if $\mu< 1+\cos^2\theta_{s+1}=2-\sin^2\theta_{s+1}$. Combining these two observations with {\bf (i)} and {\bf (ii)} in the theorem tells us that $T_\mu$ is convergent to $P_{U\cap V}$ with a rate smaller than $\cos^2\theta_{s+1}$ if and only if $\mu\in (1,2-\sin^2\theta_{s+1})$. Moreover, the optimal rate $\frac{1-\sin^2\theta_{s+1}}{1+\sin^2 \theta_{s+1}}$ of $T_\mu$ is obtained at $\mu=\frac{2}{1+\sin^2 \theta_{s+1}}$ due to \eqref{e:m2}, \eqref{e:m3}, and \eqref{e:m4}.

{\bf Case 2:} $p+q\ge n$. We may find some $k\in \NN$ such that $n^\prime:=n+k> p+q$. Define $U^\prime:=U\times\{0_k\}\subset \RR^{n^\prime}$,  $V^\prime:=V\times\{0_k\}\subset \RR^{n^\prime}$, and $T^\prime_\mu=(1-\mu)\Id+\mu P_{U^\prime}P_{V^\prime}$. It is clear that $1\le p=\dim U^\prime\le \dim V^\prime=q$ and $p+q< n^\prime$. Observe from Definition~\ref{d:PA} that the principal angles between $U^\prime$ and $V^\prime$ are the same with the ones between $U$ and $V$. Moreover, we have $P_{U^\prime}=\begin{pmatrix}P_U &0\\ 0 &0_k\end{pmatrix}$, $P_{V^\prime}=\begin{pmatrix}P_V &0\\ 0 &0_k\end{pmatrix}$, and thus
\begin{equation}\label{si}
T^\prime_\mu=\begin{pmatrix}T_{\mu}&0\\ 0 & (1-\mu)I_k\end{pmatrix}.
\end{equation}
Since $q\le n-1$, there is some $x\in \RR^n\setminus\{0\}$ such that $P_Vx=0$. It follows that $Tx=0$, and thus we have $0\in \sigma (T)$ and then $1-\mu\in \sigma(T_\mu)$. If $T_\mu$ is convergent, Fact~\ref{t:limitp} tells us that $-1<1-\mu\le 1$, i.e., $\mu\in [0,2)$. Conversely, if $\mu\in [0,2)$ we have $T^\prime_\mu$ is convergent due to Case 1. This together with \eqref{si} ensures that  $T_\mu$ is also convergent. Hence $T_\mu$ is convergent if and only if $\mu\in[0,2)$.

To verify the convergence rate of $T_\mu$, suppose further that $\mu\in (0,2)$. We note that $\sigma(T_\mu)=\sigma (T^\prime_\mu)$, which implies in turn that $\gg(T_\mu)=\gg(T^\prime_\mu)$. It follows from Case 1 that $T^\prime_\mu$ in \eqref{si} is convergent to $P_{U^\prime\cap V^\prime}=\begin{pmatrix}P_{U\cap V}&0\\0&0_k\end{pmatrix}$ with the convergence rate $\gg(T^\prime_\mu)$. This together with \eqref{si} yields
\[
\|T^n_\mu-P_{U\cap V}\|\le \|(T^\prime_\mu)^n-P_{U^\prime\cap V^\prime}\|.
\]
Thus $\gg(T^\prime_\mu)=\gg(T_\mu)$ is the convergence rate of $T_\mu$ by also Theorem~\ref{t:linear}. The analysis of $\gg(T^\prime_\mu)$ in {\bf (i)} and {\bf(ii)} in Case 1 also guides us to verify  {\bf (i)} and {\bf(ii)} for $\gg(T_\mu)$ in Case 2. Hence the proof is complete. \end{proof}

%Futhermore, suppose also that $1-\mu\sin^2\theta_k\neq 1-\mu$ and $1-\mu\sin^2\theta_{k+1}=\ldots=1-\mu\sin^2\theta_{q}=\neq 1-\mu$ for some $k\in\{s+1,\ldots,p\}$, i.e., $\theta_k\neq \frac{\pi}{2}$ while $\theta_{k+1}=\ldots=\theta_p= \frac{\pi}{2}$, we also obtain that
%\[
%\ker(T_\mu-(1-\mu)\Id)=\ker(T_\mu-a_\mu\Id)^2=D(0_{s}\times \RR^{k-s}\times 0_{n-k}),
%\]

Next we study another kind of relaxation of the the map $T=P_UP_V$, that is
\begin{equation}\label{e:PRAP}
S_\mu:=P_U((1-\mu)\Id+\mu P_V)=(1-\mu)P_U+\mu P_UP_V;
\end{equation}
{see also \cite{Luke} for a similar form,}
 which will give us a better optimal rate. Since the proof is similar to the one of Theorem~\ref{t:relaxedI} above, we only sketch the main steps.

\begin{theorem}[partial relaxed alternating projection]\label{t:relaxedII} The map $S_\mu:=P_U((1-\mu)\Id+\mu P_V)=(1-\mu)P_U+\mu P_U P_V$ is convergent if and only if $\mu\in [0, \frac{2}{\sin^2\theta_p})$ {with the convention $\frac{1}{0}=\infty$}. Moreover, the following assertions hold:

{\bf (i)} If $\mu\in (0,\frac{2}{\sin^2 \theta_{s+1}+\sin^2\theta_p}]$, then $S_\mu$ is convergent to $P_{U\cap V}$ with the optimal convergence rate $\gg(S_\mu)=1-\mu\sin^2\theta_{s+1}$.

{\bf (ii)} If $\mu\in (\frac{2}{\sin^2 \theta_{s+1}+\sin^2\theta_p}, \frac{2}{\sin^2\theta_p})$, then $S_\mu$ is convergent to $P_{U\cap V}$ with the optimal convergence rate $\gg(S_\mu)=\mu\sin^2\theta_{p}-1$.

Consequently, when $\mu\neq 0$, $S_\mu$ is convergent to $P_{U\cap V}$ with the optimal convergence rate smaller than $\cos^2\theta_{s+1}=c_F^2(U,V)$ if and only if $\mu\in (1,\frac{2-\sin^2\theta_{s+1}}{\sin^2\theta_p})$. Furthermore, $S_\mu$ attains the smallest convergence rate $\frac{\sin^2\theta_p-\sin^2\theta_{s+1}}{\sin^2\theta_{s+1}+\sin^2\theta_p}$ at $\mu= \frac{2}{\sin^2 \theta_{s+1}+\sin^2\theta_p}$.
\end{theorem}
\begin{proof} We separate the proof into two main cases as below:

{\bf Case 1:} $p +q < n$ with $1\le p=\dim U\le q=\dim V\le n-1$. It follows from \eqref{e:proj} and \eqref{e:PUV} that there is some orthogonal  matrix $D\in \RR^{n\times n}$ such that
\begin{eqnarray}\label{e:S3}
S_\mu=D\begin{pmatrix}
(1-\mu)I_p+\mu C^2 &\mu CS&0\\
0&0_p &0\\
0 &0 & 0_{n-2p}
\end{pmatrix} D^*=D\begin{pmatrix}
I_p-\mu S^2 &\mu CS&0\\
0&0_p &0\\
0 &0 & 0_{n-2p}
\end{pmatrix} D^*.
\end{eqnarray}
Hence we have
\begin{equation}\label{e:S20}
\sigma(S_\mu)=\{1-\mu\sin^2\theta_k|\; k=1,\ldots,p\}\cup\{0\}.
\end{equation} Suppose that $S_\mu$ is convergent, we get from Fact~\ref{t:limitp} that
\begin{eqnarray}\label{inq3}
-1< 1-\mu\sin^2\theta_p\quad \mbox{and}\quad 1-\mu\sin^2\theta_{s+1}\le 1.
\end{eqnarray}
Since $\theta_{s+1}=\theta_F\neq 0$ by Proposition~\ref{p:PAF}, the latter gives us that   $\mu\in[0,\frac{2}{\sin^2 \theta_p})$. Conversely, suppose that $\mu\in[0,\frac{2}{\sin^2 \theta_p})$, we have
\begin{equation}\label{mu1}
1=1-\mu\sin^2\theta_1=\cdots=1-\mu\sin^2\theta_s\ge 1-\mu\sin^2\theta_{s+1}\ge\cdots\ge 1-\mu\sin^2\theta_p>-1.
\end{equation}
If $\mu=0$ then $S_\mu=P_U$ is always convergent. If $\mu>0$ and $s=0$, it is clear that $1\notin\sigma(S_\mu)$ by \eqref{e:S20}. Thanks to Fact~\ref{t:limitp}, we have $S_\mu$ is convergent. If $\mu>0$ and $s>0$, it is similar to the corresponding part of Theorem~\ref{t:relaxedI} that $1\in \sigma(S_\mu)$ is semisimple. Combining \eqref{mu1} with
Fact~\ref{t:limitp} gives us that $S_\mu$ is convergent. Thus $S_\mu$ is convergent if and only if $\mu\in[0,\frac{2}{\sin^2\theta_p})$.

To verify {\bf(i)} and {\bf(ii)}, assume further that $\mu\in (0,\frac{2}{\sin^2\theta_p})$. Let us claim  that $S_\mu$ is convergent to $P_{U\cap V}$. Via the explicit form of $S_\mu$ in \eqref{e:S3}, we can easily check that
\[
\Fix S_\mu=\ker(S_\mu-\Id)=D((\RR^s)^*\times 0_{1\times(n-s)})^*=\ker(S^*_\mu-\Id)=\Fix S^*_\mu.
\]
Note also from  \eqref{e:UV} that
\[
U\cap V=\Fix P_{U\cap V}=D((\RR^s)^*\times 0_{1\times(n-s)})^*.
\]
It follows that $\Fix S_\mu=\Fix S^*_\mu=U\cap V$. Thanks to Corollary~\ref{c:non}, we have $S_\mu$ is convergent to $P_{U\cap V}$.

Next we justify the qualitative characterizations in {\bf(i)} and {\bf(ii)}. Observe from \eqref{e:S20} and  \eqref{mu1} that
\begin{equation}\label{mu2}
\gg(S_\mu)=\max\{|1-\mu\sin^2\theta_{s+1}|,|1-\mu\sin^2\theta_p|\}.
\end{equation}
Note also that
\begin{equation}\label{mu3}
(1-\mu\sin^2\theta_{s+1})^2-(1-\mu\sin^2\theta_p)^2=\mu(\sin^2\theta_p-\sin^2\theta_{s+1})[2-\mu(\sin^2\theta_{s+1}+\sin^2\theta_{p})].
\end{equation}

{\em Subcase a:} $\sin\theta_p=\sin\theta_{s+1}$, i.e., $\theta_{s+1}=\theta_{s+2}=\cdots=\theta_p$. Hence we have $\sigma(S_\mu)=\{1-\mu\sin^2\theta_s,1-\mu\sin^2\theta_{s+1},0\}$ and $\gg(S_\mu)=|1-\mu\sin^2\theta_{s+1}|$. Moreover, it is easy to check that $c_\mu:=1-\mu\sin^2\theta_{s+1}$ is semisimple by showing that $\ker (S_\mu-c_\mu\Id)=\ker (S_\mu-c_\mu\Id)^2$.

{\em Subcase b:} $\sin\theta_p\neq \sin\theta_{s+1}$, i.e., $\sin\theta_p> \sin\theta_{s+1}$. We continue the proof by taking into account three different cases as follows.

{\em Subsubcase b1:} $\mu\in (0,\frac{2}{\sin^2 \theta_{s+1}+\sin^2 \theta_p})$. Then we have from \eqref{mu3} that $|1-\mu\sin^2\theta_{s+1}|>|1-\mu\sin^2\theta_p|$, which gives us that $\gg(S_\mu)=|1-\mu\sin^2\theta_{s+1}|$ by \eqref{mu2}. Moreover, note that
\begin{equation}\label{inq4}
c_\mu=1-\mu\sin^2\theta_{s+1}> 1-\frac{2}{\sin^2\theta_{s+1}+\sin^2\theta_p}\sin^2\theta_{s+1}=\frac{\sin^2\theta_p-\sin^2\theta_{s+1}}{\sin^2\theta_{s+1}+\sin^2\theta_p}>0.
\end{equation}
Thanks to the structure of $S_\mu$ in \eqref{e:S3}, we may check that $c_\mu$ is semisimple. Thus $c_\mu=\gg(S_\mu)$ is the optimal convergence rate of $S_\mu$ by Theorem~\ref{t:linearII}.

{\em Subsubcase b2:} $\mu=\frac{2}{\sin^2\theta_{s+1}+\sin^2\theta_p}$. Thus
\begin{equation}\label{inq5}
\gg(S_\mu)=|1-\mu\sin^2\theta_{s+1}|=|1-\mu\sin^2\theta_p|=\frac{\sin^2\theta_p-\sin^2\theta_{s+1}}{\sin^2\theta_{s+1}+\sin^2\theta_p}> 0.
\end{equation}
We can check that $c_\mu=1-\mu\sin^2\theta_{s+1}$ and $d_\mu:=1-\mu\sin^2\theta_p$ are semisimple in this case via Fact~\ref{f:ss}. This together with Theorem~\ref{t:linearII} tells us that $\gg(S_\mu)= 1-\mu\sin^2\theta_{s+1}=\frac{\sin^2\theta_p-\sin^2\theta_{s+1}}{\sin^2\theta_{s+1}+\sin^2\theta_p}$ is the optimal linear rate of $S_\mu$.

{\em Subsubcase b3:} $\mu\in(\frac{2}{\sin^2\theta_{s+1}+\sin^2\theta_p},\frac{2}{\sin^2\theta_p})$.  It follows from \eqref{mu3} that $|1-\mu\sin^2\theta_{s+1}|<|1-\mu\sin^2\theta_p|$, which yields $\gg(S_\mu)=|1-\mu\sin^2\theta_{p}|$ by \eqref{mu2}. Moreover,
observe that
\begin{equation}\label{inq6}
\mu\sin^2\theta_{p}-1> \frac{2}{\sin^2\theta_{s+1}+\sin^2\theta_p}\sin^2\theta_{p}-1=\frac{\sin^2\theta_{p}-\sin^2\theta_{s+1}}{\sin^2\theta_{s+1}+\sin^2\theta_p}>0.
\end{equation}
We also have  $d_\mu=1-\mu\sin^2\theta_p$ is semisimple via Fact~\ref{f:ss}. Thanks to Theorem~\ref{t:linearII}, $\gg(S_\mu)=\mu\sin^2\theta_{p}-1$ is the optimal convergence rate of $S_\mu$.

Combining Subsubcase b1 and Subsubcase b2 gives us {\bf (i)}. Furthermore, Subsubcase b3 exactly verifies {\bf (ii)}.  The last part of the theorem is indeed  a direct consequence of {\bf(i)} and {\bf(ii)}. The proof of the theorem for Case 1 is complete.

{\bf Case 2:} $p+q\ge n$. Then we find some $k\in \NN$ such that $n^\prime:=n+k> p+q$ and define $U^\prime:=U\times\{0_k\}\subset\RR^{n^\prime}$,  $V^\prime:=V\times\{0_k\}\subset\RR^{n^\prime}$, and $S^\prime_\mu=(1-\mu)P_{U^\prime}+\mu P_{U^\prime}P_{V^\prime}$. It is clear that $1\le p=\dim U^\prime\le \dim V^\prime=q$ and $p+q< n^\prime$. Moreover, we also have
\[
S^\prime_\mu=\begin{pmatrix}S_{\mu}&0\\ 0 & 0_k\end{pmatrix},
\]
which shows that $S^\prime_\mu$ is convergent if and only if $S_\mu$ is convergent.  The rest of the proof is quite similar to the corresponding one in Theorem~\ref{t:relaxedI}.
\end{proof}

\begin{remark} It is clear that the optimal rate $\frac{\sin^2\theta_p-\sin^2\theta_{s+1}}{\sin^2\theta_{s+1}+\sin^2\theta_p}$ of $S_\mu$ is  smaller than the one $\frac{1-\sin^2\theta_{s+1}}{1+\sin^2\theta_{s+1}}$ of $T_\mu$ in Theorem~\ref{t:relaxedI}. Note further from the above theorem  that $S_2=P_UR_V$ with $R_V:=2P_V-\Id$, which is known  as the {\em reflection-projection} method \cite{BK,Cegielski} is convergent to $P_{U\cap V}$ if and only if $2<\frac{2}{\sin^2\theta_p}$, i.e., $\theta_p<\frac{\pi}{2}$. When this case is fulfill, the optimal rate of the reflection-projection method is $\max\{|1-2\sin^2\theta_{s+1}|,|1-2\sin^2\theta_p|\}$ by \eqref{mu2}. Besides the definition of $\theta_{s+1},\theta_p$ in Definition~\ref{d:PA} and Definition~\ref{d:F}, we may also obtain $\theta_{s+1},\theta_p$ in following formulas
\begin{equation}\label{e:th}
\cos^2\theta_{s+1}=\|P_UP_V-P_{U\cap V}\|\qquad \mbox{and}\qquad \sin^2\theta_p=\|P_U-P_UP_V\|^2=\|P_U-P_UP_VP_U\|
\end{equation}
 from \eqref{e:proj}, \eqref{e:PUV}, and \eqref{e:UV}.
\end{remark}

\begin{remark}[finite termination] From Theorem~\ref{t:relaxedI}, observe that the map $T_\mu$ has the convergence rate $0$, i.e., it will always terminate after finite powers if and only if $\theta_{s+1}=\frac{\pi}{2}$ and $\mu=1$. Similarly, we get from Theorem~\ref{t:relaxedII} that  $S_\mu$  has the convergence rate $0$ if and only if  $\mu= \frac{2}{\sin^2 \theta_{s+1}+\sin^2\theta_p}$ and $\theta_{s+1}=\theta_p$. The latter condition is clearly satisfied when $\dim(U\cap V)=p-1$ and $\mu=\frac{1}{\sin^2\theta_{s+1}}$; e.g., $U$ and $V$ are two different lines passing the origin in $\RR^2$, or  $U$ is a  line in $\RR^3$ and $V$ is a hyperplane in $\RR^3$ with $U\not\subset V$, or $U$ and $V$ are two different hyperplanes in $\RR^3$, etc.

\end{remark}

\subsection{Convergence rate of the generalized Douglas-Rachford method}
Convergence rate of many specific matrices relating to Douglas-Rachford operator
\begin{equation}\label{e:dr}
R:=P_{U}P_V+P_{U^\perp}P_{V^\perp}=\frac{R_UR_V+\Id}{2}=\frac{R_{U^\perp}R_{V^\perp}+\Id}{2}
\end{equation}
 has been discussed in  \cite{DZ}. One of the particular cases there is the so-called {\em generalized Douglas-Rachford operator} $R_\mu$ defined by
\begin{equation*}
R_\mu :=(1-\mu)\Id +\mu R.
%\end{align}
%\end{subequations}
%where
%\begin{subequations}
%\begin{align}
%R &:=P_{U}P_V+P_{U^\perp}P_{V^\perp}=\frac{R_UR_V+\Id}{2}=\frac{R_{U^\perp}R_{V^\perp}+\Id}{2}.
%\end{align}
\end{equation*}

Convergence rate of this mapping has been obtained in \cite{DZ} under an additional condition $U\cap V=\{0\}$. In the following result we  give a complete characterization of the convergence of this map and  also show that the condition $U\cap V=\{0\}$ can be relaxed.

\begin{theorem}\label{t:DR} The map $R_\mu$ is convergent if and only if $\mu\in [0,2)$. Moreover, the following assertions hold:

{\bf (i)} $R_\mu$ is normal.

{\bf (ii)} If $\mu\in (0,2)$ then $R_\mu$ is convergent to $P_{\Fix R}=P_{(U\cap V)\oplus(U^\perp\cap V^\perp)}$ with the optimal convergence rate $\gg(R_\mu)=\sqrt{\mu(2-\mu)\cos^2\theta_{s+1}+(1-\mu)^2}$, where $s:=\dim (U\cap V)$.
\end{theorem}
\begin{proof} As proceeded in the proof of Theorem~\ref{t:relaxedI} and Theorem~\ref{t:relaxedII}, we consider two major cases as below.

{\bf Case 1.} $p+q< n$. By using the expressions of \eqref{e:PUV}, we easily establish that
\begin{eqnarray}\label{69b}
\begin{array}{ll}
R_\mu&\disp=D\begin{pmatrix}
C^2+(1-\mu)S^2& \mu CS & 0  &0 \\
-\mu CS & C^2+(1-\mu)S^2 & 0 &0\\
0 & 0 & (1-\mu)I_{q-p} &0\\
0 & 0 & 0& I_{n-p-q}
\end{pmatrix}D^*\\
&\disp=D\begin{pmatrix}
I_p-\mu S^2& \mu CS & 0  &0 \\
-\mu CS & I_p-\mu S^2 & 0 &0\\
0 & 0 & (1-\mu)I_{q-p} &0\\
0 & 0 & 0& I_{n-p-q}
\end{pmatrix}D^*;
\end{array}
\end{eqnarray}
 see also a similar form on \cite[page~14]{DZ}. It is easy to check that $R^*_\mu R_\mu=R_\mu R^*_\mu$, i.e., $R_\mu$ is normal. Thus {\bf(i)} is satisfied. We may get from the above format and {the block determinant formula, c.f., \cite[page 475]{Meyer}} that
\begin{align*}
\sigma(R_\mu)=\left\{\begin{array}{ll}\{\cos^2\theta_k+(1-\mu)\sin^2\theta_k\pm \mathrm{i}\mu\cos\theta_k\sin\theta_k|\; k=1,\ldots,p\}\cup\{1\}\quad&\mbox{if}\quad q=p, \\
\{\cos^2\theta_k+(1-\mu)\sin^2\theta_k\pm \mathrm{i}\mu\cos\theta_k\sin\theta_k|\; k=1,\ldots,p\}\cup\{1\}\cup\{1-\mu\}\quad&\mbox{if}\quad q>p,\end{array}\right.
\end{align*}
where $\mathrm{i}:=\sqrt{-1}$. For any $k=1, \ldots,p$, we have
\[\begin{array}{ll}
\big|1-\mu\sin^2\theta_k\pm \mathrm{i}\mu\cos\theta_k\sin\theta_k\big|&=\sqrt{(1-\mu\sin^2\theta_k)^2+\big[\mu\cos\theta_k\sin\theta_k\big]^2}\\
&=\sqrt{\big[\mu\cos^2\theta_k+(1-\mu)]^2+\mu^2\cos^2\theta_k(1-\cos^2\theta_k)}\\
&=\sqrt{\mu(2-\mu)\cos^2\theta_k+(1-\mu)^2}.
\end{array}
\]
Suppose further that $R_\mu$ is convergent. Then we get from Fact~\ref{t:limitp} that
\[
\mu(2-\mu)\cos^2\theta_{s+1}+(1-\mu)^2\le 1,
\]
which yields $\mu(2-\mu)(1-\cos^2\theta_{s+1})\ge 0$ and thus $\mu\in[0,2]$, since $\cos^2\theta_{s+1}<1$. Next let us consider three particular subcases of $\mu$.

{\em Subcase a.} $\mu=2$. Then all eigenvalues of $R_\mu$ have magnitude $1$. By Fact~\ref{t:limitp}, we have
\begin{equation}\label{lam}
1-\mu\sin^2\theta_k\pm \mathrm{i}\mu\cos\theta_k\sin\theta_k=1\quad \mbox{for all}\quad k=1,\ldots,p,
\end{equation}
which implies in turn that  $\sin\theta_{s+1}\cos\theta_{s+1}=0$ and thus $\theta_{s+1}=\frac{\pi}{2}$, since $\sin\theta_{s+1}>0$ by Proposition~\ref{p:PAF}. It follows  that
\[
1-\mu\sin^2\theta_{s+1}\pm \mathrm{i}\mu\cos\theta_{s+1}\sin\theta_{s+1}=-1,
\]
 which contradicts \eqref{lam}. Hence when $\mu=2$, $R_\mu$ is not convergent.

{\em Subcase b:} $\mu=0$. It is obvious that $R_\mu=\Id$ is convergent to $\Id$ with rate $0$.

{\em Subcase c:} $0<\mu<2$. By Propodition~\ref{p:PAF} we have
\begin{eqnarray}\label{ee1}
\begin{array}{ll}
1&\disp=\sqrt{\mu(2-\mu)\cos^2\theta_1+(1-\mu)^2}=\cdots =\sqrt{\mu(2-\mu)\cos^2\theta_s+(1-\mu)^2}\\
&\disp>\sqrt{\mu(2-\mu)\cos^2\theta_{s+1}+(1-\mu)^2}\ge\sqrt{\mu(2-\mu)\cos^2\theta_{s+2}+(1-\lm)^2}\\
&\disp\ge\cdots \ge\sqrt{\mu(2-\mu)\cos^2\theta_{p}+(1-\mu)^2}\ge  |1-\mu|.\label{70c}
\end{array}
\end{eqnarray}
Since $R_\mu$ is normal, it follows from  Fact~\ref{t:limitp} and Corollary~\ref{c:non} that $R_\mu$ is convergent. Hence we have $R_\mu$ is convergent if and only if $\mu\in [0,2)$.

It remains to verify {\bf(ii)} in this case. Suppose that $\mu\in (0,2)$, we get from the normality of $R_\mu$ and  Theorem~\ref{normalI} that $\gg(R_\mu)=\sqrt{\mu(2-\mu)\cos^2\theta_{s+1}+(1-\mu)^2}$ (by \eqref{ee1}) is the optimal convergence rate of $R_\mu$ and that $R_\mu$ is convergent to $P_{\Fix R_\mu}=P_{\Fix R}$. Moreover, we have $\Fix R=(U\cap V)\oplus(U^\perp\cap V^\perp)$ by \cite[Proposition~3.6]{BCNPW}. This ensures {\bf(ii)} and thus completes the proof of the theorem for Case 1.\\[-2ex]

{\bf Case 2:} $p+q\ge n$. Similarly to the proof of Theorem~\ref{t:relaxedI} and Theorem~\ref{t:relaxedII}, we find $k>0$ such that $n+k:=n^\prime>p+q$. Define further that $U^\prime:=U\times \{0_k\}\subset \RR^{n^\prime}$, $V^\prime:=V \times \{0_k\} \subset \RR^{n^\prime}$, and $R^\prime_\mu=(1-\mu)\Id+\mu[P_{U^\prime}P_{V^\prime}+P_{(U^\prime)^\perp}P_{(V^\prime)^\perp}]$. It is easy to verify that
\begin{equation}\label{72}
R^\prime_\mu=\begin{pmatrix} R_\mu &0 \\ 0 &I_k\end{pmatrix}.
\end{equation}
Note from Case 1 that $R^\prime_\mu$ is normal, and so is $R_\mu$.  Morever, we get from \eqref{72} that $R_\mu$ is convergent if and only if $R^\prime_\mu$ is convergent with the same rate. The analysis of the convergence of $R^\prime_\mu$ in Case 1 justifies all the statement of the theorem in this case. The proof is complete.
\end{proof}

\begin{remark} (1). Unlike the relaxed alternating projection methods studied in Theorem~\ref{t:relaxedI} and ~\ref{t:relaxedII}, convergence rate of the (over and under) relaxation of the Douglas-Rachford algorithm is always bigger than the original one due to
\[
\gg(R_1)=\cos\theta_{s+1}\le \sqrt{\mu(2-\mu)\cos^2\theta_{s+1}+(1-\mu)^2}=\gg(R_\mu)\quad \mbox{for all}\quad \mu\in[0,2).
\]
Moreover, it is worth mentioning here that Theorem~\ref{t:DR} also tells us that $R_2=R_UR_V$, which is known as \emph{reflection-reflection} method will never be convergent in the case of two nontrivial subspaces with $1\le \dim U,\dim V\le n-1$.

(2). For the convergence rate of the Douglas-Rachford method on a general Hilbert space, see \cite{BCNPW}.
\end{remark}

\section{A nonlinear approach to the alternating projection method}

Throughout this section, we also suppose that $U$ and $V$ are two subspaces of $\RR^n$ with $1\le p=\dim U\le \dim V=q\le n-1$. From Theorem~\ref{t:relaxedII},  we know that the map $S_\mu$ \eqref{e:PRAP} obtains its smallest rate $\frac{\sin^2\theta_p-\sin^2\theta_{s+1}}{\sin^2\theta_{s+1}+\sin^2\theta_p}$ at $\mu=\frac{2}{\sin^2\theta_{s+1}+\sin^2\theta_p}$. {This rate is smaller than the optimal rate of $T_\mu$ and $T$.} However, it is not trivial to determine $\theta_{s+1}$ and $\theta_p$ to construct $\mu=\frac{2}{\sin^2\theta_{s+1}+\sin^2\theta_p}$ for $S_\mu$ especially with big dimensions of $U$ and $V$; see Definition~\ref{d:PA}, Definition~\ref{d:F}, and  \eqref{e:th}. In this section we introduce a simple {\em nonlinear} mapping, by using the idea of a line search \cite{BDHP,GK,GPR} for the map $S_\mu$, so that the iterative sequence given by this nonlinear
mapping is linearly convergent to the projection on $U\cap V$ with the same optimal rate mentioned above.  One may think of this mapping as the partial relaxed alternating projection with an adaptive parameter $\mu(x)$ depending on each iteration period. This is a technique employed for other iterative methods, see, e.g., \cite{BC2011,BCK,Cegielski,CS,censor07,elfving}.

\begin{definition} Define the map $B_T$ with $T=P_UP_V$ by
\begin{equation}\label{acce}
B_T(x):=P_U((1-\mu_x)x+\mu_x P_Vx)=(1-\mu_x)P_Ux+\mu_x P_UP_Vx,
\end{equation}
where
\begin{eqnarray}\label{mux}
\mu_x:=\left\{\begin{array}{ll}\disp\frac{\la P_Ux-P_UP_Vx,x\ra}{\|P_Ux-P_UP_Vx\|^2}\quad&\mbox{if}\quad P_Ux-P_UP_Vx\neq 0\\
\disp 1\quad & \mbox{if}\quad P_Ux-P_UP_Vx=0.
 \end{array}\right.
\end{eqnarray}
\end{definition}
\begin{remark} In \cite{BCK,BDHP,GK}, an {\em accelerated} mapping of $T$ is introduced by using the line-search \cite{GPR} as
 \begin{equation}\label{AT}
A_T(x):=(1-\lm_x)x+\lm_x P_UP_Vx,
\end{equation}
where
\begin{eqnarray}\label{lmx}
\lm_x=\left\{\begin{array}{ll}\disp\frac{\la x-P_UP_Vx,x\ra}{\|x-P_UP_Vx\|^2}\quad&\mbox{if}\quad x-P_UP_Vx\neq 0\\
\disp 1\quad & \mbox{if}\quad x-P_UP_Vx=0.
 \end{array}\right.
\end{eqnarray}
It is worth noting that  $\mu_x=\lm_x$ and $B_Tx=A_Tx$ when $x\in U$.
\end{remark}

Set $M:=U\cap V$. The proof of the following convenient fact can be found in \cite[Lemma~9.2]{Deutsch}
\begin{equation}\label{pemu}
P_UP_M=P_MP_U=P_VP_M=P_MP_V=P_M.
\end{equation}
The main result in this section is Theorem~\ref{t:new}, before proving it we provide two useful lemmas.

\begin{lemma}For each $x\in\RR^n$ and $y\in U\cap V$ we have
\begin{align}\label{e:inf}
\min_{\mu\in \RR}\|(1-\mu)P_Ux+\mu P_UP_Vx-y\|= \|B_Tx-y\|.
\end{align}
Moreover, $\mu_x$ given in \eqref{mux} is the unique minimizer when $P_Ux-P_UP_Vx\neq 0$.
\end{lemma}
\begin{proof}
{When $P_Ux=P_UP_Vx$, inequality \eqref{e:inf} is trivial. Now suppose that $P_Ux\neq P_UP_Vx$ and note that
\begin{align}
\|(1-\mu)P_Ux+\mu P_UP_Vx-y\|^2&=\|(1-\mu)(P_Ux-y)+\mu (P_UP_Vx-y)\|^2\nonumber\\
&=(1-\mu)\|P_Ux-y\|^2+\mu\|P_UP_Vx-y\|^2-\mu(1-\mu)\|P_Ux-P_UP_Vx\|^2.\label{e:00}
\end{align}
This is a quadratic in $\mu$ and thus attains its minimum at the following unique minimizer
\begin{eqnarray}\label{e:000}
\begin{array}{ll}
\mu&\disp=\frac{1}{2}\frac{\|P_Ux-y\|^2-\|P_UP_Vx-y\|^2+\|P_Ux-P_UP_Vx\|^2}{\|P_Ux-P_UP_Vx\|^2}\\
&\disp=\frac{\la P_Ux-y,P_Ux-P_UP_Vx\ra}{\|P_Ux-P_UP_Vx\|^2}.\end{array}
\end{eqnarray}
Since $y\in U\cap V$, we derive that
\begin{equation}\label{e:01}
\la y, P_Ux-P_UP_Vx\ra=\la y,P_UP_{V^\perp}x\ra=\la P_Uy,P_{V^\perp}x\ra=\la y, P_{V^\perp}x\ra=0.
\end{equation}
Moreover, note that $\la P_Ux,P_Ux-P_UP_Vx\ra=\la x, P_Ux-P_UP_Vx\ra$. This together with \eqref{e:00}, \eqref{e:000} and \eqref{e:01} tells us that the left-hand side of \eqref{e:inf} attains its minimum at $\mu_x$ in \eqref{mux}. We verify \eqref{e:inf} and complete the proof of the lemma.}
\end{proof}

\begin{lemma} For any $\mu\in \RR$ and $x\in U$ we have
\begin{equation}\label{e:Sn}
S_\mu x-P_{U\cap V}x=((1-\mu)P_U+\mu P_UP_VP_U-P_{U\cap V})(x-P_{U\cap V}x),
\end{equation}
where $S_\mu=(1-\mu) P_U+\mu P_UP_V$ defined in {\rm Theorem~\ref{t:relaxedII}}. Moreover, when $\mu=\frac{2}{\sin^2\theta_{s+1}+\sin^2\theta_p}$ with $s=\dim(U\cap V)$ and $\theta_{s+1},\theta_p$ found in {\rm Definition~\ref{d:PA}}, we have
\begin{equation}\label{e:norm}
\|(1-\mu)P_U+\mu P_UP_VP_U-P_{U\cap V}\|=\frac{\sin^2\theta_p-\sin^2\theta_{s+1}}{\sin^2\theta_{s+1}+\sin^2\theta_p}.
\end{equation}
\end{lemma}
\begin{proof}  Set $M:=U\cap V$. For any $x\in U$ we get from \eqref{pemu} that
\begin{align*}
((1-\mu)P_U+\mu P_UP_VP_U-P_M)(x-P_Mx)&\disp=((1-\mu)P_U+\mu P_UP_VP_U)x-P_Mx-P_Mx+P_M^2x\\
&\disp=(1-\mu)P_Ux+\mu P_UP_Vx-P_Mx\\
&\disp=S_\mu x-P_Mx,
\end{align*}
which verifies \eqref{e:Sn}. To justify \eqref{e:norm}, without loss of generality, suppose that $p+q<n$ (otherwise, we follow the trick used in the proof of Case 2 of Theorem~\ref{t:relaxedII}).  It is easy to check from \eqref{e:proj} and \eqref{e:UV} that
\begin{align*}
(1-\mu)P_U+\mu P_UP_VP_U-P_M&\disp=D\begin{pmatrix}(1-\mu)I_p+\mu C^2-\begin{pmatrix}I_s &0\\ 0&0_{p-s}\end{pmatrix} &0 \\ 0 &0_{n-p}\end{pmatrix}D^*\\
&\disp=D\begin{pmatrix}0_s & & & 0\\ &1-\mu\sin^2\theta_{s+1} & &\\ &\qquad\qquad\qquad \ddots & &\\
&  &1-\mu\sin^2\theta_{p} &\\
0 &  & & 0_{n-p}\end{pmatrix}D^*
\end{align*}
for some orthogonal  matrix $D\in \RR^{n\times n}$. When $\mu=\frac{2}{\sin^2\theta_{s+1}+\sin^2\theta_p}$, we get that
\[\begin{array}{ll}
\|(1-\mu)P_U+\mu P_UP_VP_U-P_M\|&\disp=\max\big\{|1-\mu\sin^2\theta_{p}|,|1-\mu\sin^2\theta_{s+1}|\big\}\\
 &=\disp \frac{\sin^2\theta_p-\sin^2\theta_{s+1}}{\sin^2\theta_{s+1}+\sin^2\theta_p}.
\end{array}
\]
This ensures \eqref{e:norm} and completes the proof of the lemma.
\end{proof}

We are ready to establish the main result of this section as follows.

\begin{theorem}\label{t:new} For any $n\in \NN$ and $x\in \RR^n$, we have
\begin{equation}\label{e:cog}
\|B^{n+1}_T(x)-P_{U\cap V}x\|\le \left[\frac{\sin^2\theta_p-\sin^2\theta_{s+1}}{\sin^2\theta_p+\sin^2\theta_{s+1}}\right]^{n}\cos^2\theta_{s+1}\|x-P_{U\cap V}x\|,
\end{equation}
where $\theta_{s+1}$ and $\theta_{p}$ are the principal angles found in {\em Definition~\ref{d:PA}}. Hence the algorithm $B_T^n(x)\to P_{U\cap V}(x)$ is at least as fast as the partial relaxed projection \eqref{e:PRAP}. Furthermore, when $x\in U$ we obtain a shaper inequality
\begin{equation}\label{e:cog2}
\|B^{n+1}_T(x)-P_{U\cap V}x\|\le \Big[\frac{\sin^2\theta_p-\sin^2\theta_{s+1}}{\sin^2\theta_p+\sin^2\theta_{s+1}}\Big]^{n+1}\|x-P_{U\cap V}x\|.
\end{equation}
\end{theorem}
\begin{proof} For any $x\in \RR^n$, define $y= P_{U\cap V}x$ and $M=U\cap V$, note that $y=P_MB_Tx$. Fix $\mu:=\frac{2}{\sin^2\theta_{s+1}+\sin^2\theta_p}$ and $\gg:=\frac{\sin^2\theta_p-\sin^2\theta_{s+1}}{\sin^2\theta_{s+1}+\sin^2\theta_p}$. We obtain
\begin{align*}
\|B^{n+1}_T(x)-P_{U\cap V}x\|&\disp=\|B_T(B^{n}_T(x))-y\|\le \|S_\mu(B^{n}_T(x))-y\| \qquad (\mbox{by } \eqref{e:inf})\\
&\disp=\|((1-\mu)P_U+\mu P_UP_VP_U-P_M)(B_T^{n}(x)-y)\| \qquad \mbox{(by  \eqref{e:Sn} and $B_T^{n}x\in U$)}\\
&\le \disp\|(1-\mu)P_U+\mu P_UP_VP_U-P_M\|\cdot \|B_T^{n}(x)-y\|\\
&= \frac{\sin^2\theta_p-\sin^2\theta_{s+1}}{\sin^2\theta_{s+1}+\sin^2\theta_p} \|B_T^{n}(x)-y\|\qquad(\mbox{by } \eqref{e:norm})\\
&\le \cdots \\
&\disp\le \gg^{n}\|B_T(x)-y\|\le \gg^{n}\|S_1(x)-y\|\qquad (\mbox{by \eqref{e:inf} again)}\\
&=\disp \gg^{n}\|(P_UP_V-P_M)(x-P_Mx)\|\qquad (\mbox{by \eqref{pemu}})\\
&\le\disp \gg^{n}\|P_UP_V-P_M\|\cdot\|x-P_Mx\|=\gg^{n}\cos^2\theta_{s+1}\|x-P_Mx\| \qquad (\mbox{by \eqref{e:th})}.
\end{align*}
This verifies \eqref{e:cog}. To justify \eqref{e:cog2}, suppose further that $x\in U$,  note that
\[
B_Tx-P_Mx=[(1-\mu_x)P_U+\mu_xP_UP_VP_U](x-P_Mx).
\]
With $y=P_Mx$, following the above inequalities gives us that
\begin{align*}
\|B^{n+1}_T(x)-P_{U\cap V}x\|&\disp\le \gg^{n}\|B_T(x)-y\|\le\gg^{n}\|S_\mu x-y\|\qquad \mbox{(by \eqref{e:inf})}\\
&\disp= \gg^n\|((1-\mu)P_U+\mu P_UP_VP_U-P_M)(x-y)\|\qquad \mbox{(by \eqref{e:Sn})}\\
&\disp\le \gg^{n+1}\|x-y\|\qquad \mbox{(by \eqref{e:norm})},
\end{align*}
which ensures \eqref{e:cog2} and completes the proof of the theorem.
\end{proof}

\begin{remark} As discussed at the beginning of this section, though the map $S_\mu$ obtains the optimal convergence rate at $\mu_0:=\frac{2}{\sin^2\theta_{s+1}+\sin^2\theta_p}$, computing $\theta_{s+1}$ and $\theta_p$ may be expensive when the dimensions of $U$ and $V$ are big. Our nonlinear map $B_T$  indeed has a similar form to $S_\mu$ and also obtains the same rate with $S_{\mu_0}$, but it is easier to compute $\mu_x$ in \eqref{mux} and hence $B_T(x)$ for any $x\in\RR^n$.

\end{remark}

The following corollary suggests a convergence rate for the accelerated map $A_T$ in \eqref{AT}. This is actually a counterpart of \cite[Theorem~3.28]{BDHP} when $T=P_UP_V$, which is not \emph{selfadjoint} as required in \cite[Theorem~3.28]{BDHP}.

 \begin{corollary} Let $T=P_UP_V$. Then for any $n\in \NN$ and $x\in \RR^n$, we have
\begin{equation}\label{e:cog3}
\|A^n_T(Tx)-P_{U\cap V}x\|\le \left[\frac{\sin^2\theta_p-\sin^2\theta_{s+1}}{\sin^2\theta_p+\sin^2\theta_{s+1}}\right]^{n}\cos^2\theta_{s+1}\|x-P_{U\cap V}x\|,
\end{equation}
where  $A_T$ is defined in \eqref{AT} and where $\theta_{s+1}$ and $\theta_{p}$ are the principal angles found in {\em Definition~\ref{d:PA}}.
 \end{corollary}
 \begin{proof} For any $x\in \RR^n$, note that $Tx\in U$, $P_{U\cap V}Tx=TP_{U\cap V}x=P_{U\cap V}x$ by \eqref{pemu}, and that $A^n_T(Tx)=B^n_T(Tx)$. Thus we get from \eqref{e:cog2} that
\begin{align*}
\|A^n_T(Tx)-P_{U\cap V}x\|&\le \Big[\frac{\sin^2\theta_p-\sin^2\theta_{s+1}}{\sin^2\theta_p+\sin^2\theta_{s+1}}\Big]^{n}\|Tx-P_{U\cap V}x\|\\
&= \Big[\frac{\sin^2\theta_p-\sin^2\theta_{s+1}}{\sin^2\theta_p+\sin^2\theta_{s+1}}\Big]^{n}\|(T-P_{U\cap V})(x-P_{U\cap V}x)\|\\
&\le \Big[\frac{\sin^2\theta_p-\sin^2\theta_{s+1}}{\sin^2\theta_p+\sin^2\theta_{s+1}}\Big]^{n}\cos^2\theta_{s+1}\|x-P_{U\cap V}x\| \qquad (\mbox{by \eqref{e:th})},
\end{align*}
 which completes the proof of the corollary.
\end{proof}

%%%%%%%%%%%%%%%%%%%%%%%%%%%%%%%%%%%%%%%%%%%%
\section{Numerical experiments}\label{s:numerical}

In this section, we compare several algorithms developed in previous sections with some classic methods for finding $P_{U\cap V}x_0$. Our test algorithms are the following
\begin{itemize}
\item $B_T$ defined in \eqref{acce};
\item $S_\mu$ with $\mu_1=\frac{2}{\sin^2\theta_F+\sin^2\theta_p}$ (the best parameter); $\mu_2=\frac{1}{\sin^2\theta_p}\in
[0,\frac{2}{\sin^2\theta_p})$; and $\mu_3=\tfrac{1}{2}+\frac{1}{\sin^2\theta_p}\in
[1,\frac{2}{\sin^2\theta_p})$ (see Theorem~\ref{t:relaxedII});
\item $T_{\mu}$ with $\mu_1=\tfrac{2}{1+\sin^2\theta_F}$ (the best parameter) and $\mu_2=1.5\in[0,2)$ (see Theorem~\ref{t:relaxedI});
\item the classic method of alternating projections (MAP);
\item the classic Douglas-Rachford method (DR).
\end{itemize}
There are (at least) two angles that might affect the convergence: $\theta_F$ (the Friedrichs angle, see~Proposition~\ref{p:PAF}) and $\theta_p$, thus we will use them to categorize the pairs of subspaces. Our numerical set up is as follows. We assume that $X=\RR^{100}$ and define $\mcX$ to be the set of all subspaces of $X$. First, we define our \emph{primary} categories based on the Friedrichs angle (in radians) as follows
\begin{subequations}
\begin{align}
W_1&:=\menge{(U,V)\in\mcX^2}{0<\theta_F<0.05};\\
W_2&:=\menge{(U,V)\in\mcX^2}{0.05\leq\theta_F<0.1};\\
W_3&:=\menge{(U,V)\in\mcX^2}{0.1\leq\theta_F<0.5};\\
\text{and}\quad
W_4&:=\menge{(U,V)\in\mcX^2}{0.5\leq\theta_F<1}.
\end{align}
\end{subequations}
Since we always have $\theta_F\leq\theta_p\leq\pi/2$, we define our \emph{secondary} categories as follows\footnote{We do not test the case $\theta_F=\theta_{s+1}=\theta_p$: in such case, it is proved in Theorems~\ref{t:relaxedII} and \ref{t:new} that $S_{\mu_1}$ and $B_T$ converge after a single step, i.e., they are the clear winners!}
\begin{equation}
Z_j:=\menge{(U,V)\in \mcX^2}
{\theta_p>\theta_F\quad\text{and}\quad
\tfrac{\theta_p-\theta_F}{\frac{\pi}{2}-\theta_F}\in[\tfrac{j-1}{5},\tfrac{j}{5})},\quad j=1,\ldots,5.
\end{equation}
Thus, there are 20 induced categories $W_i\cap Z_j$ for $i=1,\ldots,4$ and $j=1,\ldots, 5$. In each $W_i\cap Z_j$, we randomly generated 5 pairs of subspaces $U$ and $V$ of $X$ such that $\dim U\leq\dim V$ and $U\cap V\neq\{0\}$. So there are $100$ pairs of subspaces. For each pair of subspaces, we choose randomly 10 starting points, each with Euclidean norm 10. This results in a total of $1,000$ instances for each algorithm. Note that the sequences to monitor are as follows

\begin{center}
\begin{tabular}{c|c}
Algorithm & sequence $(z_{n})$ to monitor\\[+1mm]
\hline
& \\[-3mm]
$B_T$ (see \eqref{acce}) & $(B_T)^n(x_0)$\\[+1mm]
\hline
& \\[-3mm]
$S_\mu$ (see Theorem~\ref{t:relaxedII})& $(S_\mu)^n(x_0)$\\[+1mm]
\hline
& \\[-3mm]
$T_\mu$ (see Theorem~\ref{t:relaxedI})& $(T_\mu)^n(x_0)$\\[+1mm]
\hline
& \\[-3mm]
MAP & $(P_UP_V)^n(x_0)$\\[+1mm]
\hline
& \\[-3mm]
DR (see~\ref{e:dr}) & $P_V(\frac{\Id+R_UR_V}{2})^n(x_0)$
\end{tabular}
\end{center}
We terminate the algorithm when the current iterate of the monitored sequence $(z_n)_\nnn$ satisfies
\begin{equation}
d_{U\cap V}(z_{n})\leq 0.01
\end{equation}
for the first time or when the number of iterations reaches $100,000$ (i.e., problem unsolved). In applications, we in general would not have access to this information but here we use it to see the true performance of these algorithms.

In Figures~\ref{fig:1}, \ref{fig:2}, and \ref{fig:3}, the horizontal axis represents the Friedrichs angle between two subspaces; and the vertical axis represents the (median) number of iterations, more specifically, the median is computed over 10 instances of one pair of subspaces.

In Figure~\ref{fig:1}, we compare $B_T$, the ``best" versions $S_{\mu_1}$ and $T_{\mu_1}$, MAP, and DR. We see that $B_T$ is generally the fastest when $\theta_F>0.02$. This can be interpreted by the fact that $B_T$ optimizes its parameter $\mu_x$ at each iteration. While when $\theta_F\leq0.02$, DR seems to be the fastest, this phenomenon has been previously observed in \cite{BCNPW}. In Figure~\ref{fig:2}, we compare $S_{\mu_i}$, $i=1,2,3$. The results suggest that the ``best" version $S_{\mu_1}$ is somewhat faster than $S_{\mu_2}$ and $S_{\mu_3}$.  In Figure~\ref{fig:3}, we compare $T_{\mu_i}$, $i=1,2$. On the contrary, it is not clear that the ``best'' version $T_{\mu_1}$ is more favorable than $T_{\mu_2}$.

\begin{figure}[H]
\centering
\includegraphics{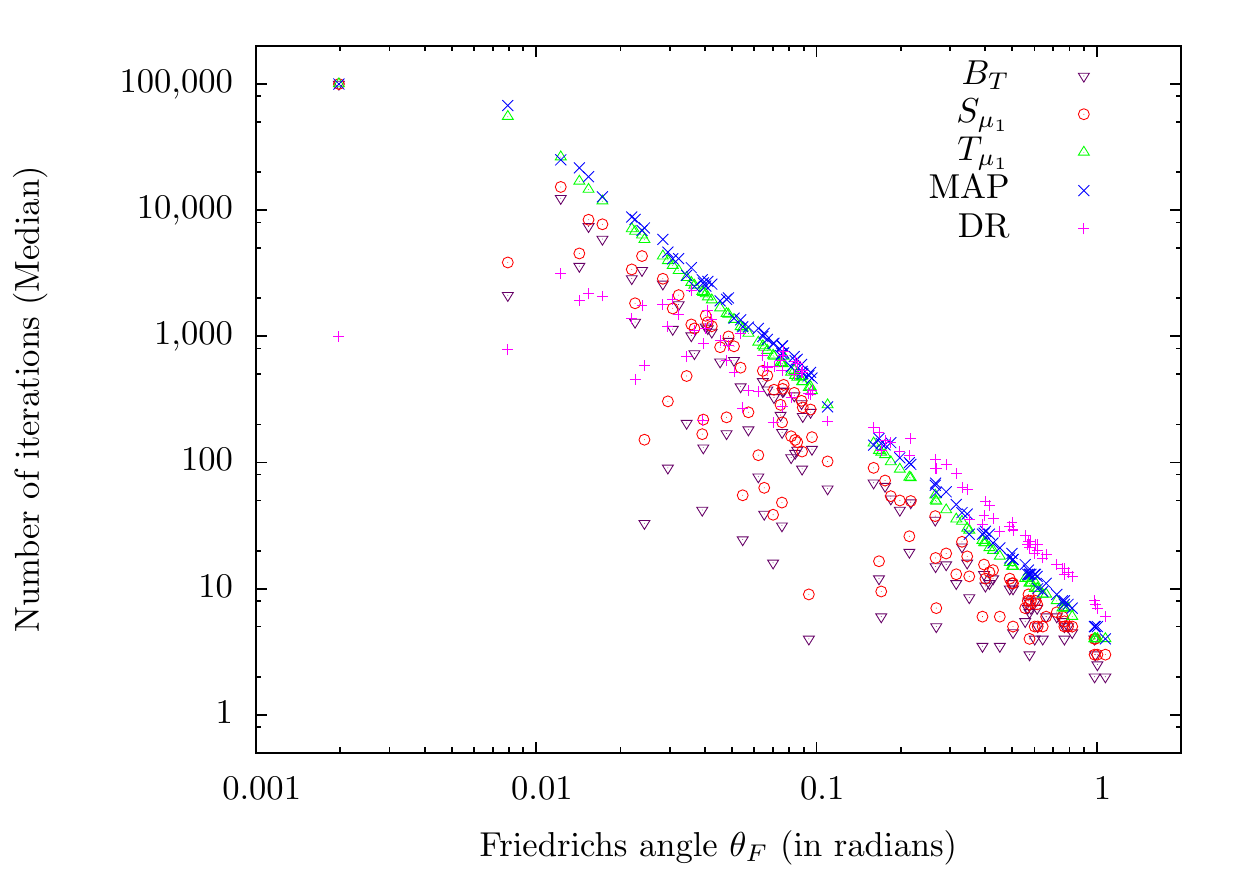}
\caption{$B_T$ is the fastest for large $\theta_F$, while DR is the fastest for small $\theta_F$.}
\label{fig:1}
\end{figure}

\begin{figure}[H]
\centering
\includegraphics{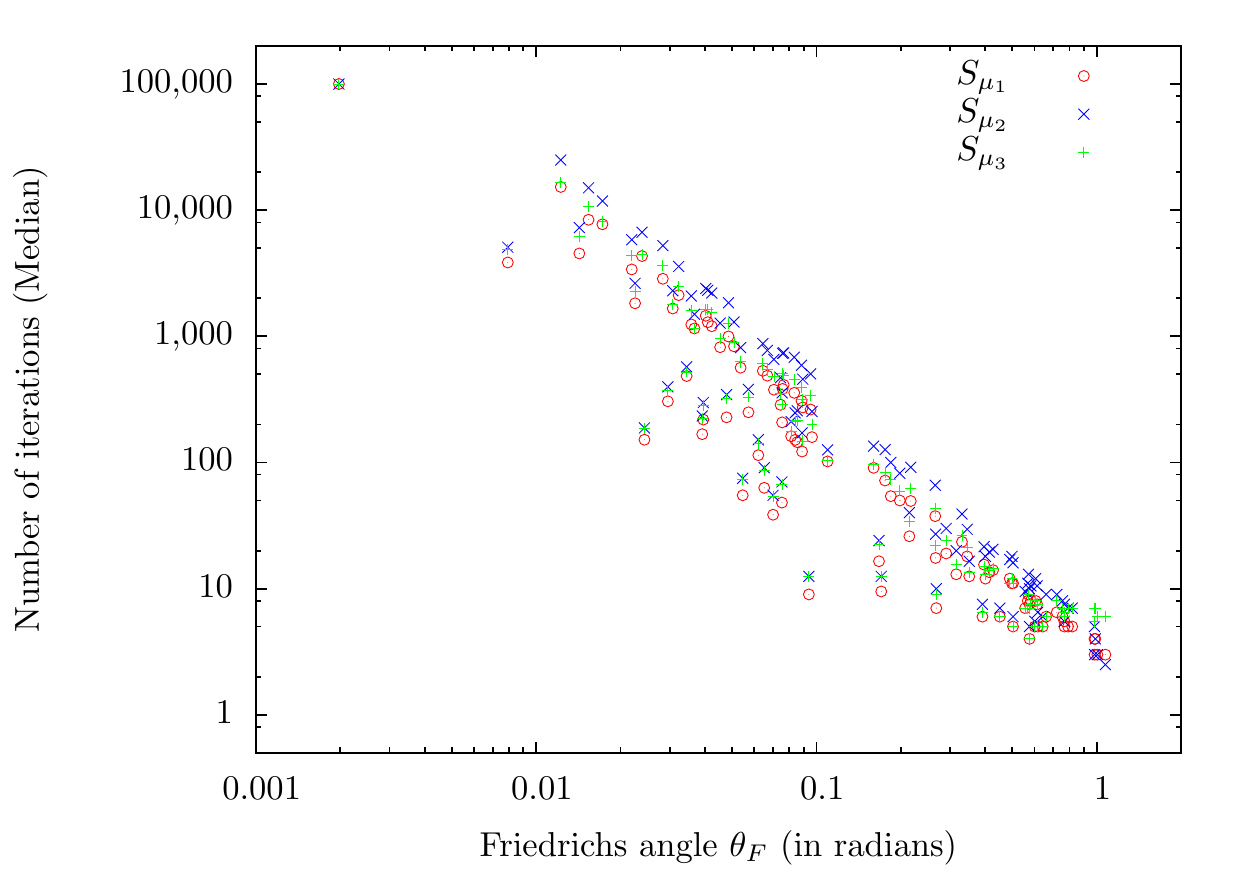}
\caption{$S_{\mu}$ with $\mu_1=\tfrac{2}{\sin^2\theta_F+\sin^2\theta_p}$ (``best"); $\mu_2=\tfrac{1}{\sin^2\theta_p}$; and $\mu_3=\tfrac{1}{2}+\tfrac{1}{\sin^2\theta_p}$.}\label{fig:2}
\end{figure}

\begin{figure}[H]
\centering
\includegraphics{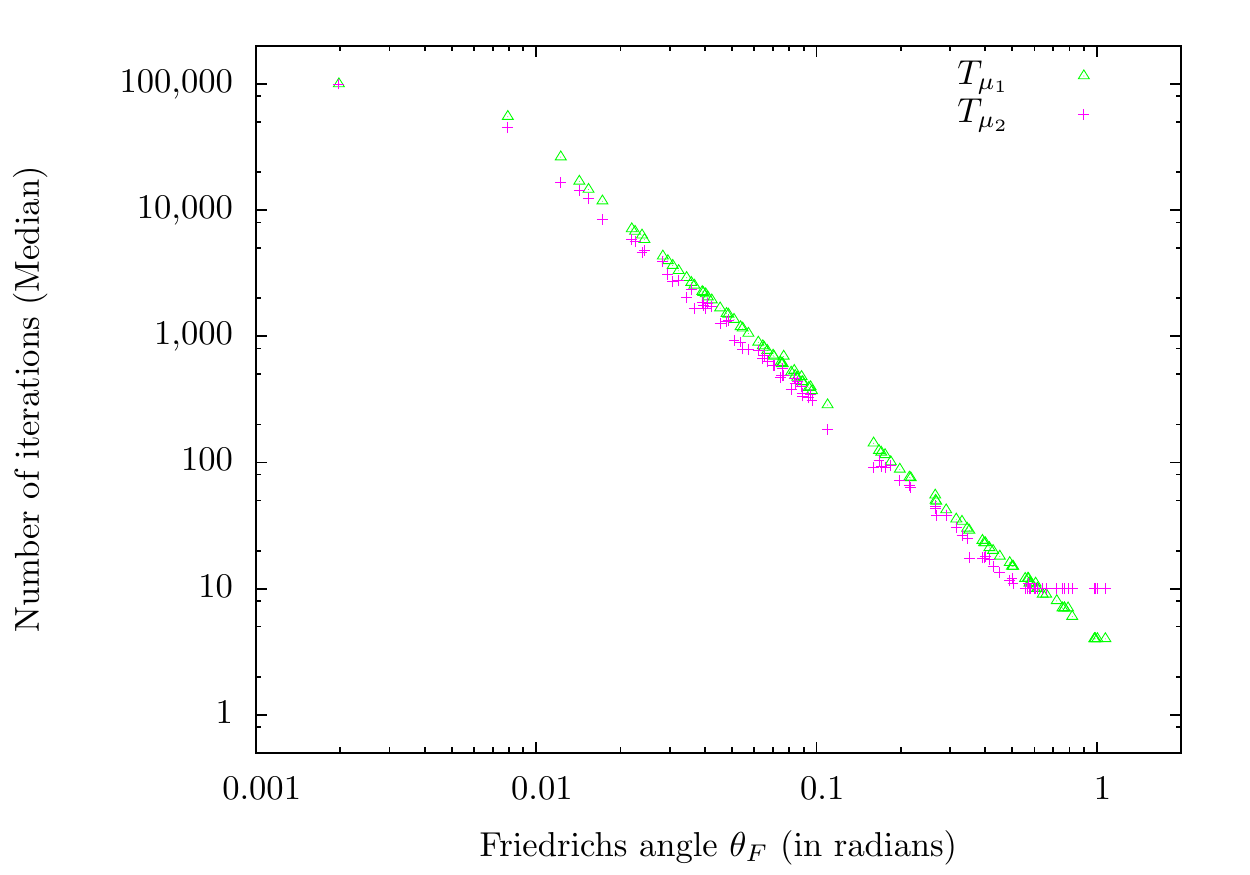}
\caption{$T_\mu$ with $\mu_1=\tfrac{1}{1+\sin^2\theta_F}$ (``best"); and $\mu_2=1.5$.}
\label{fig:3}
\end{figure}

Finally, in Table~\ref{tb:1}, for each primary category $W_i$, we record the median, the mean, and the standard deviation of the number of iterations required for the algorithms to terminate. The table clearly supports these observations above. In general, the results suggest that all algorithms are more preferable than MAP.

\begin{table}[H]
\centering
\begin{tabular}{|c|c|c|c|c|c|}
\hline
\multicolumn{2}{|c|}{} & & & & \\[-3mm]
\multicolumn{2}{|c|}{Primary category}&
$W_1$ &
$W_2$ &
$W_3$ &
$W_4$ \\[+1mm]
\hline
\multicolumn{2}{|c|}{} & & & & \\[-3mm]
\multicolumn{2}{|c|}{$\begin{aligned}
\text{Number of}\\ \text{instances}
\end{aligned}$}&
250& 250 & 250 & 250 \\[+1mm]
%%% BT %%%%
\hline & & & & & \\[-3mm]
\multirow{5}{*}{$B_T$}
& Median & 1139 & 169 & 13.5 & 5 \\[+1mm]
\cline{2-6}& & & & &\\[-3mm]
& Mean & 6002.5 & 206.7 & 22.9 & 5.1 \\[+1mm]
\cline{2-6}& & & & &\\[-3mm]
& Std & 19437.1 & 163.8 & 21.6 & 2 \\[+1mm]
%%% S_mu_best %%%
\hline & & & & & \\[-3mm]
\multirow{4}{*}{$S_{\mu_1}$}
& Median & 1404 & 226.5 & 16 & 5 \\[+1mm]
\cline{2-6}& & & & &\\[-3mm]
& Mean & 6586.8 & 260.1 & 27.8 & 5.8 \\[+1mm]
\cline{2-6}& & & & &\\[-3mm]
& Std & 19396.5 & 195.7 & 26 & 2.1 \\[+1mm]
%%% S_mu %%%
\hline & & & & & \\[-3mm]
\multirow{4}{*}{$S_{\mu_2}$}
& Median & 2318.5 & 359.5 & 24.5 & 7 \\[+1mm]
\cline{2-6}& & & & &\\[-3mm]
& Mean & 8096.6 & 417.5 & 43.5 & 7.6 \\[+1mm]
\cline{2-6}& & & & &\\[-3mm]
& Std & 19657.3 & 326.9 & 42.5 & 3.5 \\[+1mm]
%%% S_mu_2 %%%
\hline & & & & & \\[-3mm]
\multirow{4}{*}{$S_{\mu_3}$}
& Median & 1697.5 & 272 & 18.5 & 7 \\[+1mm]
\cline{2-6}& & & & &\\[-3mm]
& Mean & 6980.3 & 307.4 & 32 & 6.7 \\[+1mm]
\cline{2-6}& & & & &\\[-3mm]
& Std & 19426.9 & 221.5 & 30.3 & 1.5 \\[+1mm]
%%% T_mu_best %%%
\hline & & & & & \\[-3mm]
\multirow{5}{*}{$T_{\mu_1}$}
& Median & 3636.5 & 611 & 42 & 9\\[+1mm]
\cline{2-6}& & & & &\\[-3mm]
& Mean & 11571 & 684.7 & 64.5 & 8.8 \\[+1mm]
\cline{2-6}& & & & &\\[-3mm]
& Std & 21298.1 & 265.9 & 59.3 & 3.3 \\[+1mm]
%%% T_1.5 %%%
\hline & & & & & \\[-3mm]
\multirow{5}{*}{$T_{\mu_2}$}
& Median & 2704.5 & 481.5 & 32.5 & 10 \\[+1mm]
\cline{2-6}& & & & &\\[-3mm]
& Mean & 9788.2 & 528.2 & 48.8 & 10.2 \\[+1mm]
\cline{2-6}& & & & &\\[-3mm]
& Std & 20599 & 223.1 & 44.2 & 0.6 \\[+1mm]
%%% MAP %%%
\hline & & & & & \\[-3mm]
\multirow{5}{*}{MAP}
& Median & 4058.5 & 722.5 & 49 & 10 \\[+1mm]
\cline{2-6}& & & & &\\[-3mm]
& Mean & 12683 & 793.1 & 74 & 10.2 \\[+1mm]
\cline{2-6}& & & & &\\[-3mm]
& Std & 22531.1 & 334.7 & 66.4 & 4.1 \\[+1mm]
%%% DR %%%
\hline & & & & & \\[-3mm]
\multirow{5}{*}{DR}
& Median & 1231 & 448.5 & 83.5 & 17.5 \\[+1mm]
\cline{2-6}& & & & &\\[-3mm]
& Mean & 1395.2 & 511.3 & 92 & 17.4 \\[+1mm]
\cline{2-6}& & & & &\\[-3mm]
& Std & 847.9 & 203.6 & 56.1 & 7.1 \\[+1mm]
\hline
%\multirow{2}{*}{$B_T$}
%& & & & &
\end{tabular}
\caption{Median, mean, and standard deviation of number of iterations.}
\label{tb:1}
\end{table}
The data and figures in this section were computed with the help of \texttt{Julia} (see \cite{Julia}) and \texttt{Gnuplot} (see
\cite{Gnuplot}).

%%%%%%%%%%%%%%%%%%%%%%%%%%%%%%%%%%%%%%%%%%%%
\section{Conclusion}\label{s:conclusion}
This paper presents a constructive study on the optimal convergence linear rate of a matrix. We give a complete characterization when the  matrix has the optimal convergence rate in term of semi-simpleness of all the subdominant eigenvalues and the unit eigenvalue.
%In this paper, we characterize the optimal convergence rate of matrices powers in terms of semisimpleness of subdominant eignvalues.
Combined with the principal angles between two subspaces,
this allows us to provide convergence analysis
for relaxed alternating projection, partial relaxed alternating projection and generalized Douglas-Rachford methods for two subspaces.
%Consequently, we apply our general results to the relaxed alternating projection and generalized Douglas-Rachford methods for two subspaces.
It turns out that the partial relaxed alternating projection method and its nonlinear version could obtain the smallest convergence rate among these ones, which are demonstrated by numerical performances. Our results not only recover
but also significantly extend currently known results in the literature.
In future research one may similarly investigate Jacobi, Gauss-Seidel, and  especially successive over-relaxation methods. Understanding further the partial relaxed alternating projection method for two sets is also an intriguing project.

\section*{Acknowledgments}

HHB was partially supported by a Discovery Grant and an Accelerator
Supplement of the Natural Sciences and Engineering
Research Council of Canada (NSERC) and by the Canada Research Chair Program. JYBC was partially supported by CNPq grants 303492/2013-9, 474160/2013-0
and 202677/2013-3 and by project CAPES-MES-CUBA 226/2012. TTAN was partially supported by a postdoctoral
fellowship of the Pacific Institute for the Mathematical Sciences
and by NSERC grants of HHB and XW.
HMP was partially supported by NSERC grants of HHB and XW. XW was partially supported by a Discovery Grant of NSERC.

% \small

\end{document}